\documentclass[11pt,reqno]{amsart}
\usepackage{dsfont}
\usepackage{graphicx}
\usepackage{latexsym}
\usepackage{amssymb}
\usepackage{amsmath}
\usepackage{enumerate}
\usepackage{mathrsfs}
\usepackage{cases}
\usepackage{footmisc}

\allowdisplaybreaks[4]

\RequirePackage[colorlinks,citecolor=blue,urlcolor=blue,backref=false]{hyperref}

\newtheorem{theorem}{Theorem}[section]
\newtheorem{prop}[theorem]{Proposition}
\newtheorem{lemma}[theorem]{Lemma}

\newtheorem{remark}[theorem]{Remark}

\usepackage{tikz}
\usepackage[english]{babel}
\usepackage[T1]{fontenc}
\usepackage[utf8]{inputenc}

\numberwithin{equation}{section}

\newcommand{\E}{\mathbb{E}}
\newcommand{\R}{\mathbb{R}}
\newcommand{\N}{\mathbb{N}}
\newcommand{\X}{\mathbb{X}}
\newcommand{\Q}{\mathbb{Q}}
\renewcommand{\P}{\mathbb{P}}
\newcommand{\md}{{\rm d}}
\newcommand{\p}{\operatorname{pow}}

\newcommand{\Var}{\operatorname{Var}}

\def\cF{\mathcal{F}}
\def\cM{\mathcal{M}}
\def\cN{\mathcal{N}}
\def\cL{\mathcal{L}}
\def\cE{\mathcal{E}}

\def\bN{\mathbf{N}}

\def \ind{\mathds{1}}

\makeatletter
\def\namedlabel#1#2{\begingroup
    #2%
    \def\@currentlabel{#2}%
    \phantomsection\label{#1}\endgroup
}
\makeatother

\usepackage{color}

\title{Gaussian approximation for Extreme Points in Laguerre tessellations}
\author{Chinmoy Bhattacharjee}

\address{Chinmoy Bhattacharjee: Universität Hamburg,
Bundesstraße 55,
20146 Hamburg, Germany}
\email{chinmoy.bhattacharjee@uni-hamburg.de}

   \author{Anna Gusakova}
   \address{Anna Gusakova:
Univeristy of M\"unster,
Orl\'eans-Ring 10,
48149 M\"unster, Germany}
\email{gusakova@uni-muenster.de}
\date{\today}

\begin{document}
\begin{abstract}
    We consider Gaussian approximation in three particular models of Poisson-Laguerre tessellations, namely, the $\beta$-, $\beta'$- and Gaussian-Voronoi tessellations. The tessellations are constructed based on inhomogeneous Poisson point processes in space-time $\R^d\times \R$, where some of the points of the process give rise to a cell in $\R^d$, known as extreme points, while the other points produce an empty cell. Using the notion of region-stabilization, we derive quantitative central limit theorems with presumably optimal rates of convergence for the number of extreme points of $\beta$-, $\beta'$- and Gaussian-Voronoi tessellations in a growing window $W_n=[-n,n]^d$ as $n\to\infty$. Our bounds improve and extend previously known results by Schreiber and Yukich (2008) for the $\beta$-model, and are the first quantitative results for the $\beta'$- and Gaussian models.
\end{abstract}

\maketitle

\section{Introduction}

A random tessellation in $\R^d$ is a locally finite collection of non-empty convex polytopes with disjoint interiors that together cover the entire space. Random tessellations are among the central objects of study in stochastic geometry and have been intensively investigated for decades. Patterns similar to those described by random tessellations naturally arise in biology, cosmology, and materials science, making them appealing models for various natural phenomena (see, for example, \cite{BPR25, LR22, Gibbs} and references therein). At the same time, tessellation models have found applications in actively developing areas of computer science and machine learning, particularly in spatial data structures and geometric inference \cite{ORT24}. For a broad overview of different random tessellation models and modeling techniques, we refer the reader to \cite{RJ25}.

\subsection*{Model.} One of the earliest and most fundamental random tessellation models is the Poisson-Voronoi tessellation. Let $\eta$ be a homogeneous Poisson point process in $\R^d$, then for any point $v\in\eta$ we define its Voronoi cell as
\[
V(v)=\{w\in\R^d\colon \|v-w\|\leq \|v'-w\|\, \forall v'\in \eta\}.
\]
The Poisson-Voronoi tessellation arises as the collection of all such Voronoi cells, which are convex polytopes (see \cite{CSKM} and \cite[Chapter 10]{SW} for further details). This model can be interpreted as a crystal growth process, where each nucleus $v\in\eta$ represents a crystal center, all crystals get activated at the same time and grow with the same speed in every direction. Each point $w\in\R^d$ belongs to the crystal that reaches it first. 

A natural generalization of this construction is the so-called \textit{Laguerre tessellation}. Here as generating points, we consider a set $A$ of pairs (seeds) $x:=(v,h)\in\R^d\times \R$ with $v\in\R^d$ being the center of its crystal (spacial coordinate) and $h$ being its activation time (time coordinate). The crystal centered at $v$ gets activated at time $h$ and subsequently grows isotropically with the same decaying speed $1/(2\sqrt{t})$ after time $t>0$ from its activation. A point $w \in \R^d$ is included in the cell of $x =(v,h) \in A$ if its
crystal reaches $w$ first (see Section \ref{sec:Laguerre} for more details). Formally, the cell (crystal) of the seed $x=(v,h)\in A$ is defined as
\begin{equation}\label{eq:cell}
C(x, A):= \{w \in \R^d: \p(w,x) \le \p(w,x') \;\; \forall x'=(v',h') \in A\},
\end{equation}
where
$$
\p(w,(v,h)) := \|w-v\|^2 + h,
$$
is the power of a point $w \in \R^d$ w.r.t.\ $(v,h)$. Note that unlike Voronoi cells, Laguerre cells need not be non-empty or contain their centers. Hence, the object that interests us is the so-called Laguerre tessellation
\begin{equation}\label{eq:Lag}
    \cL(A)=\{C(x, A)\colon x\in\eta, {\rm int}\, C(x, \eta)\neq \emptyset\},
\end{equation}
the collection of all cells with a non-empty interior. We call a point $x \in A$ an \textit{extreme point} if it forms a cell in $\cL(A)$, that is, when ${\rm int}\, C(x, A)\neq \emptyset$.  It should be noted that $\cL(A)$ is not always a tessellation in the sense of the definition given above, since for a general set $A$, cells of $\cL(A)$ may not necessarily be bounded and the collection of cells may not necessarily be locally-finite. However, if $A$ satisfies some additional mild conditions, these issues do not arise \cite{Sch93}. We refer the reader to \cite[Chapter 17]{BookAlgGeom}, \cite{RJ25} and the references therein for further background on Laguerre tessellations. For applications of this specific model of random tessellations, we refer the reader to, for example, \cite{BPR25} and references therein.

In this paper, we study the probabilistic properties of random Laguerre tessellations, where the generating set is given by a Poisson point process in space-time. More precisely, let $\eta$ be a Poisson point process in $\R^d \times E$, $E\subset \R$, equipped with the Borel $\sigma$-algebra and intensity measure $\Q=\lambda \otimes F$, where $\lambda$ denotes the Lebesgue measure on $\R^d$ and $F$ is a locally finite Borel measure on $E$. Note here that in general, we let our time space be the whole $\R$ or an open half interval $(-\infty,b)$, $b\in\R$. The random Laguerre tessellation $\cL(\eta)$ for $F$ a probability measure on $(-\infty,0)$ satisfying some natural integrability assumptions was first introduced and studied in \cite{LZ08}. More recently, the case when $F$ has a density $f$ with respect to Lebesgue measure satisfying some additional mild conditions was treated in \cite{GiWL25}, while nonparametric inference in the case when $F$ is locally finite measure on $(0,\infty)$ was considered in \cite{vdJJV25}. 

Among the possible choices for the density function $f$, the following three families have attracted particular attention due to their tractability and rich geometric behaviour:
\begin{enumerate}
    \item For $\beta \in (-1,\infty)$, let $\eta_\beta$ denote a Poisson point process in $\R^d \times [0,\infty)$ with intensity measure $\Q_{\beta}$ given by
\[
\gamma\,c_{d,\beta}h^\beta \md h \md v,\qquad \gamma>0,\quad c_{d,\beta}={\Gamma({d\over 2}+\beta+{3\over 2})\over \pi^{d+1\over 2}\Gamma(\beta+1)}.
\]
\item For $\beta\in(d/2+1,\infty)$, let $\eta'_\beta$ denote a Poisson point process in $\R^d \times (-\infty,0)$ with intensity measure $\Q'_{\beta}$ given by 
\[
\gamma c'_{d,\beta}(-h)^{-\beta} \md h \md v,\qquad \gamma>0,\quad c'_{d,\beta}={\Gamma(\beta)\over \pi^{d+1\over 2}\Gamma(\beta-{d+1\over 2})}.
\]
\item Let $\widetilde \eta$ be a Poisson point process in $\R^d \times \R$ with intensity measure $\widetilde\Q$ given by 
\[
e^{h} \md h \md v.
\]
\end{enumerate} 

\begin{figure}[t]
\centering
\includegraphics[width=0.3\columnwidth]{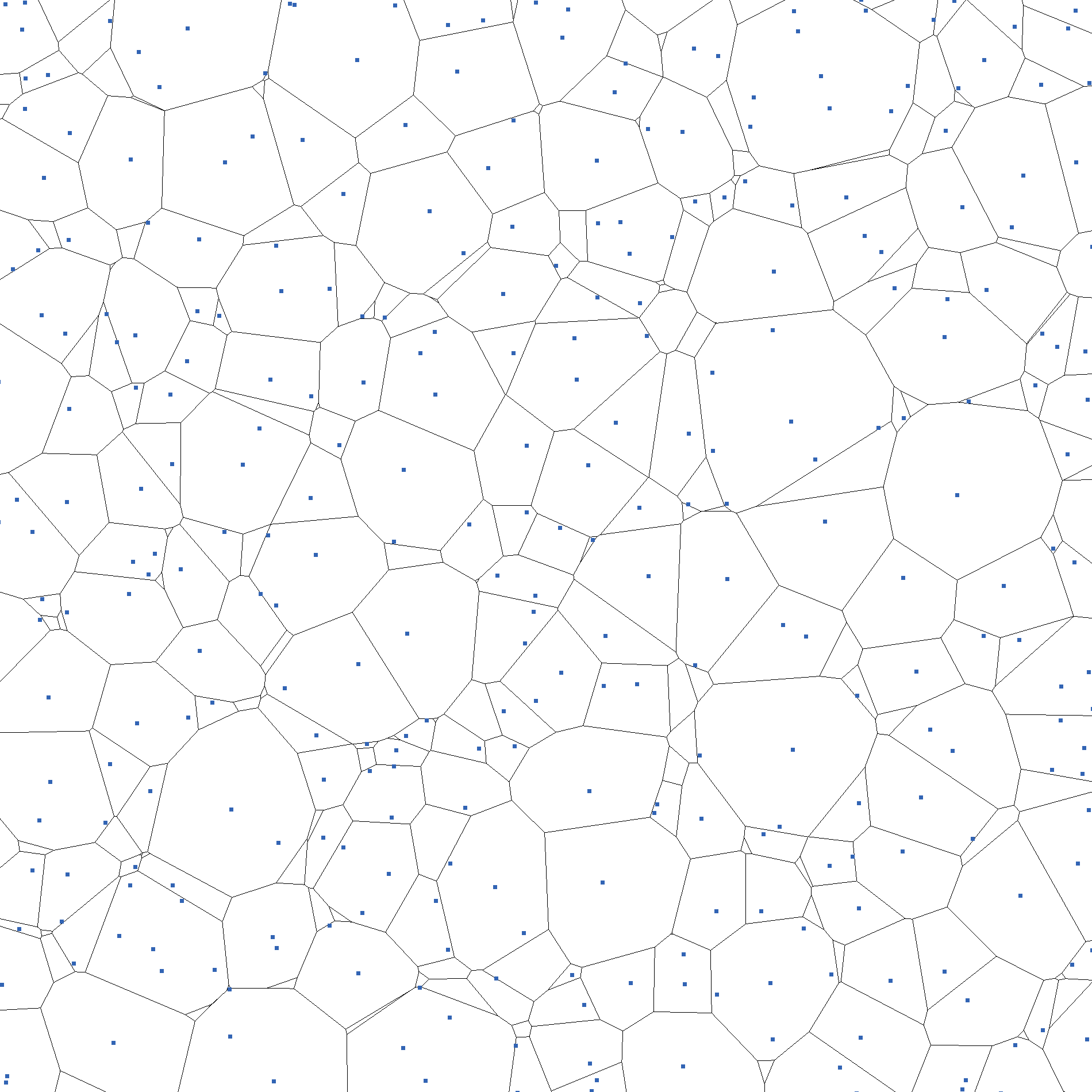}
\quad
\includegraphics[width=0.3\columnwidth]{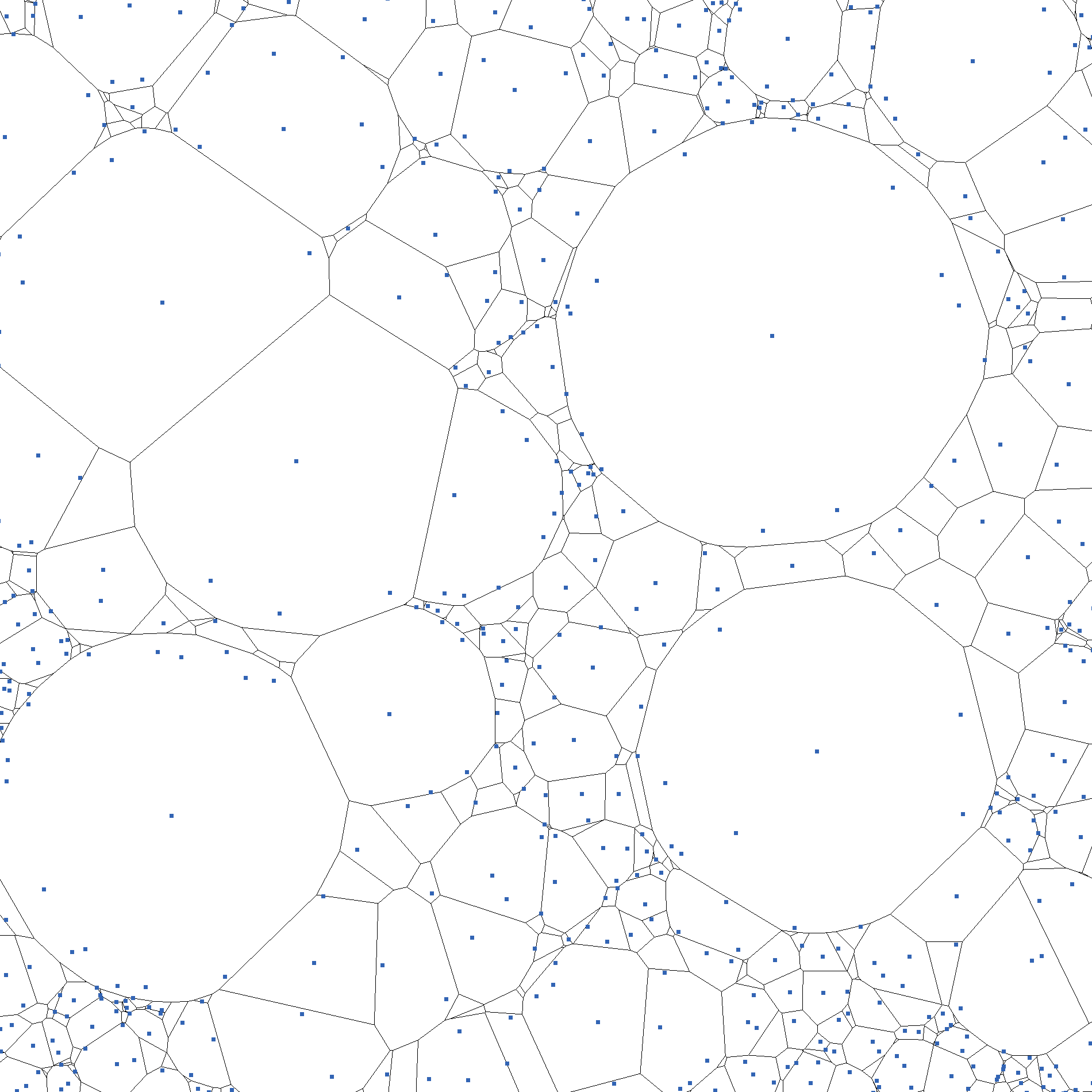}
\quad
\includegraphics[width=0.3\columnwidth]{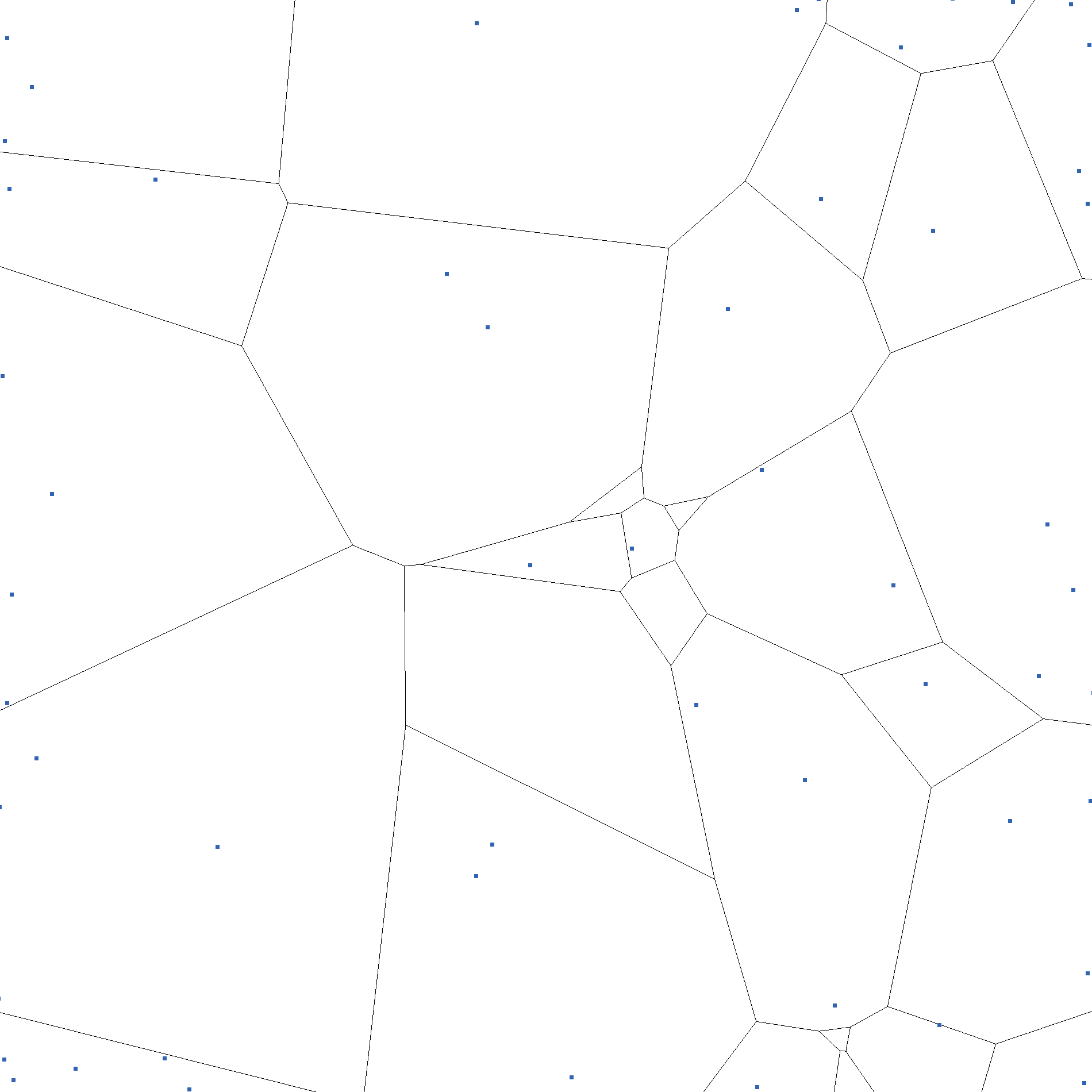}
\caption{Left: $\beta$-Voronoi tessellation in $\R^2$ with $\beta=5$. Middle: $\beta'$-Voronoi tessellation in $\R^2$ with $\beta=2.5$. Right: Gaussian-Voronoi tessellation in $\R^2$. The corresponding centers of growth are marked in blue.} \label{fig:BetaModels}
\end{figure}

The corresponding Laguerre tessellations appear to have certain special features allowing for a more detailed probabilistic analysis (see \cite[Section 5.3]{GiWL25} for more details). We call $\cL(\eta_{\beta})$, $\cL(\eta'_{\beta})$ and $\cL(\widetilde \eta)$ the $\beta$-, the $\beta'$- and the Gaussian-Voronoi tessellations, respectively, and for simplicity in the rest of the paper will often refer to them as $\beta$, $\beta'$ and Gaussian model, respectively. These models have been introduced and studied extensively in a sequence of works \cite{GKT22, GKT22b, GKT23}. The postfix \textit{Voronoi} in the name is motivated by the fact that $\beta$-Voronoi tessellation has the same distribution as the Poisson-Voronoi tessellation intersected with a lower-dimensional subspace \cite{GKT24}. Furthermore, the Gaussian-Voronoi tessellation arises as a scaling limit of both $\beta$- and $\beta'$- Voronoi tessellations when $\beta\to\infty$ (see \cite{GKT22b} and \cite{GKT24} for more details).

\subsection*{Main result.}  Our goal in this paper is to study distributional approximations for the number of extreme points in a random Laguerre tessellation whose spatial coordinates lie within a growing observation window in $\R^d$. More precisely, for $n \in \N$ and $W_n:= [-n,n]^d$, let $F_n(\eta)$ denote the number of extreme points with their spatial coordinate lying in $W_n$, given by
\begin{equation}\label{eq:Fn}
F_n(\eta) : = \sum_{(v,h) \in \eta} \ind (v \in W_n) \ind ((v,h) \text{ forms a cell in $\cL(\eta)$}).
\end{equation}
Note that a cell associated with a point in $W_n$ may extend partially or even entirely outside this window. 

The study of central limit theorems in similar birth-growth models was initiated in \cite{Bar}, where a more general class for the speed of growth was considered. Subsequently, in \cite{SY}, the first quantitative results providing rates for Gaussian convergence in the case of the $\beta$-model with $\beta\ge 0$ was established, albeit with an extraneous logarithmic factor. More recently, for the processes $\eta_{\beta}$ or $\widetilde\eta$, a central limit theorem for $F_n(\eta)$ without any rates of convergence was established in \cite{GKT23}, relying on a general central limit theorem for absolutely regular stationary random fields (or tessellations) \cite[Theorem 4.9]{Heinrich} (see also \cite[Theorem 3.1]{Hein94}), and requiring good bounds for mixing coefficients of corresponding random tessellations. 
However, the corresponding bounds for mixing coefficients in the case of $\beta'$-model (while being presumably optimal) appears to be insufficient for applying the results in \cite{Heinrich}. The main aim of the current paper is to fill these gaps, and provide a quantitative result for all the three models with presumably optimal rates of convergence.

To state our main result, we recall two standard probability metrics used to measure the distance between distributions: the Wasserstein and the Kolmogorov distances. The Wasserstein distance between (the distributions of) two real-valued random
variables $X$ and $Y$ is given by
$$
d_W(X,Y):= \sup_{h \in \operatorname{Lip}_1} |\E\; h(X) - \E \; h(Y)|,
$$
where $\operatorname{Lip}_1$ denotes the class of all Lipschitz
functions $h: \R \to \R$ with Lipschitz constant at most one. The Kolmogorov distance, obtained by restricting the test functions to indicators of half-lines, is given by
$$
d_K(X,Y):= \sup_{t \in \R} |\P(X \le t) - \P(Y \le t)|.
$$
In the following main result of our paper, we establish a quantitative central limit theorem for $F_n$ for all the three models, providing a presumably optimal rate of convergence, which is optimal in the Kolmogorov case.

\begin{theorem}\label{thm:1} Let $\eta$ be either (a) $\eta_\beta$ with $\beta>-1$, or (b) $\eta_\beta'$ with $\beta>5d+1$, or (c) $\widetilde \eta$. Then there exists $0 < C_1\le  C_2<\infty$ depending only on $d,\beta,\gamma$ such that
$$
C_1 n^{d} \le \operatorname{Var}(F_n(\eta)) \le C_2 n^d.
$$
Moreover, for $N \sim \cN(0,1)$, there exists a constant $C \in (0,\infty)$ depending only on $d,\beta,\gamma$ such that for all $n \in \N$,
    $$
    d\Big(\frac{F_n(\eta) - \E F_n(\eta) }{\sqrt{\operatorname{Var} F_n(\eta) }}, N\Big) \le \frac{C}{n^{d/2}}
    $$
    for $d \in \{d_W,d_K\}$. Moreover, the bound on the Kolmogorov distance is of optimal order.
\end{theorem}

Let us briefly compare our results with the existing relevant literature. First, our quantitative bound in Theorem \ref{thm:1} removes the extraneous logarithmic factor that appears in the bounds in \cite{SY} for the $\beta$-model with $\beta\ge 0$ and extends the result to $\beta\in (-1,0)$. Moreover, our result appears to be the first quantitative CLT for the $\beta'$- and the Gaussian models. We also note here that \cite[Theorem 1.2]{SY} established exact variance asymptotics for the specific case of the $\beta$-model with $\beta\ge 0$, while our variance estimates in Theorem \ref{thm:1} apply to all the three models, though they only provide lower and upper bounds of matching order in $n$. Finally, we mention the work \cite[Corollary 1.1]{CY}, which utilizes the notion of localization to study a related functional in the Gaussian model, particularly its variance asymptotics.

In proving Theorem \ref{thm:1}, there are two main challenges that we tackle. First, the seeds in the model typically have long-range interactions, making proving a qualitative CLT using classical stabilization methods difficult. We instead use a more recent approach based on the notion of {\em region-§stabilization} introduced in \cite{BM22}. Secondly, in all our examples, the measure $F$ on the time space is not integrable (or, as in the case of the $\beta'$-model, not even locally integrable). This makes it very challenging to optimally estimate some of the integrals that typically arise in stabilization related quantitative CLT bounds. An important intermediate step in order to control the so-called \textit{region of stabilization} (see Section \ref{sec:cltBM} for details) is to obtain good tail bounds for the coverage time $T(x,\eta)$, which is the last point in time, when the cell $C(x,\eta)$ sees any growth (after this time, the cell stays unchanged). These tail bounds are established in Proposition \ref{lm:tailbounds}, and are of independent interest. 

 \begin{remark}[Voronoi tessellation] We point out here that when $\eta$ is a homogeneous Poisson point process with intensity $1$, we may still define the Laguerre tessellation based on $\eta$ by taking the activation time of every point $v\in\eta$ to be $h=0$. In this case, the obtained tessellation coincides with the classical Voronoi tessellation and every point $(v,0)$, $v\in\eta$ is an extreme point. Hence, $F_n(\eta)=\sum_{v\in\eta}\ind(v\in W_n)$, which is a Poisson random variable with parameter $(2n)^d$, and hence satisfies a CLT as $n\to\infty$.
 \end{remark}

\begin{remark}[Range of $\beta$ in Theorem \ref{thm:1} for the $\beta'$ model]\label{rem:1}
    The effects of a long-range interaction and non-integrable $F$ are especially amplified in the $\beta'$ model. In this case, the exponential factor appearing in the upper bound to the probability of a seed $x_1$ influencing whether another seed $x_2$ is extreme or not, behaves like a positive constant for points activated near time $-\infty$ (see \eqref{eq:tailbeta'} for more details), which makes the analysis particularly difficult. This in particular doesn't let us prove our result in the $\beta'$-model for $d/2+1<\beta \le 5d+1$. Our lower limit $5d+1$ is very likely suboptimal.

\end{remark}

\begin{remark}[Comparison with Johnson-Mehl tessellations] In \cite{BMT}, a quantitative Gaussian approximation was proved for the so-called exposed seeds in a Johnson-Mehl (JM) birth-growth model, where seeds appear randomly in $\R^d$ at
random times in $[0,\infty)$ and start growing instantaneously in all directions with a random speed, with the triplet of location, birth time, and growth speed modeled according to a Poisson point process. 

We first note the models are different in the follwing two senses: first, in the JM model, the birth-times are in $[0,\infty)$ where for us, we allow it to be in $\R$. Secondly, In the JM model, the cells grow at a constant speed throughout its lifetime, while in our model, we have a parabolic growth, which means that the speed of growth slows down with time. 

However, the main difference in the JM model and our model is that, a seed is exposed in the JM model if it is not covered by another cell at the time of its birth. In contrast, in our model, a seed may be extreme, even if it is covered at the time of its birth, but only emerges outside of the union of other growing cells at a later time after its birth. This is an important difference since it introduces interaction at a much larger scale. Thus, even though the results in \cite{BMT} also uses the notion of stabilization regions, in our work we require a different and more complex region of stabilization involving the random coverage time $T$ (see \eqref{eq:covertime1} for a definition) of our cells.
\end{remark}

\begin{remark}[Extensions to weghted functionals of extreme points]
A natural next step to Theorem \ref{thm:1} would be to ask whether similar quantitative CLTs hold for certain weighted functionals of the extreme points. In particular, one can study functionals of the form
    \begin{equation}\label{eq:Fn'}
        F_n'(\eta, w): = \sum_{(v,h) \in \eta} \ind (v \in W_n) \ind ((v,h) \text{ forms a cell in $\cL(\eta)$}) \, w((v,h),\eta),
    \end{equation}
    where $w$ is weight function depending on the seed $(v,h)$ and the whole point collection $\eta$. In particular consider
    \begin{align*}
        w((v,h),\eta)={1\over d+1-k}\sum_{F\in\cF_{k}(C((v,h),\eta))}\ind(z(F)\in W_n),
    \end{align*}
    where $\cF_k(P)$ is the set of all $k$-dimensional faces of $P$ and $z(P)$ is the center of the circumball of polytope $P$ (center of $P$). In this case if the tessellation $\cL(\eta)$ is normal the functional $F_n'(\eta,w)$ counts the number of $k$-dimensional faces of $\cL(\eta)$ with their centers in $W_n$. One could also choose another center function $z$, for example $z(P)$ being the lexicographically smallest vertex of $P$.
   Another prominent example would be $w((v,h),\eta))={\rm vol}(C((v,h),\eta))$. In this case $F_n'(\eta,w)$ is the total volume of all cells of $\cL(\eta)$ whose nuclei $v$ fall within $W_n$.

   Our approach using region stabilization continues to apply to functionals of the form \eqref{eq:Fn'}. In particular, for both the above choices of weights (or for any weight that is completely determined by the cell of the point $(v,h)$), the so-called region of stabilization remains the same as we consider in this paper. Thus, almost all of our computations go through, provided (1) $w$ satisfies certain conditions required for Theorem \ref{thm:KolBd} below, especially having a finite $(4+p)$-moment uniformly in $(v,h)$, and (2) one can obtain a suitable variance lower bound for $F_n'(\eta, w)$. While this might be possible to achieve for both the examples above, it is a challenging task and will require significantly more work, and hence they are not the focus of the present work.
\end{remark}

\subsection*{Structure of the paper.} The rest of the paper is organized as follows. In Section 2 we collect the results on quantitative CLT for stabilizing Poisson functionals. In Section 3 we prove a few auxiliary results for Poisson-Laguerre tessellations, including tail estimates for last time when the cell of given seed $x$ sees any growth (coverage time) and the lower bounds for the variance. Section 4 is devoted to the proof of Theorem \ref{thm:1}.

\section{Preliminaries}

\subsection{Notation} Throughout the manuscript, $B^d(x,r)$ denotes the $d$-dimensional closed Euclidean ball with center $x\in\R^d$ and radius $r>0$. Given a set $A\subset \R^d$, $\partial A$ stands for its boundary. By $|A|$ we denote the cardinality of an at most countable set $A$. We let 
\[
\kappa_d={\pi^{d\over 2}\over \Gamma({d\over 2}+1)}
\]
denote the volume of the unit ball in $\R^d$. We write $A\lesssim B$ to mean that there exists a constant $C$ independent of $n$, such that $A\le C B$.

\subsection{A quantitative CLT for stabilizing Poisson functionals}\label{sec:cltBM}

Let $(\X,\cF)$ be a measurable space with a $\sigma$-finite measure $\Q$ and a measurable semi-metric $\md:\X \times \X \to [0,\infty)$. For $s \ge 1$, let $\eta_s$ be a Poisson process on $\X$ with intensity measure $s\Q$. Denote by $\bN$ the set of all simple and locally finite counting measures on $\X$, equipped with an appropriate $\sigma$-algebra $\mathscr{N}$. In \cite{BM22}, a quantitative CLT as $s\to\infty$ was proved for statistics of the form 
\begin{equation}\label{eq:scoresum}
	H_s(\eta_s)= \sum_{x \in \eta_s} \xi_s(x,\eta_s),
\end{equation}
where $\xi_s: \X \times \bN \to \R$ is a measurable \textit{score function}, that roughly represents the local contribution of the point $x$ in the overall statistic $H_s$. For certain stabilizing functionals, it may happen that the so-called \textit{radius of stabilization} (see e.g.\ \cite{LSY19}) is too large to obtain meaningful Gaussian approximation bounds. In such a context, it is often useful to consider general stabilization regions for the scores.

For our purposes, we fix $s=1$ throughout, and hence drop all such subscripts $s$. Specifically, we let $\xi:=\xi_1$, $\eta=\eta_1$ and $H:=H_1$ in \eqref{eq:scoresum}. Moreover, we make the following additional assumptions. Below, we write $0$ for the empty point configuration, while for two configurations $\cM_1$ and $\cM_2$, we write $\cM_1 \le \cM_2$ when $\cM_1$ is contained in $\cM_2$ as a point configuration. Also, for a subset $A \subset \X$ and $\cM \in \bN$, we write $\cM_{A}$ to denote the restriction of the configuration $\cM$ to the set $A$.

\begin{enumerate}
    \item[(A0)] \textit{Monotonicity:} Assume that if $\xi(x, \mathcal{M}_1)=\xi(x, \mathcal{M}_2)$ for some $\mathcal{M}_1,\mathcal{M}_2 \in \mathbf{N}$
with $0\neq \mathcal{M}_1\leq \mathcal{M}_2$, then
\begin{equation}
	\label{eq:ximon}
	\xi(x, \mathcal{M}_1)=\xi(x, \mathcal{M}') \quad 
	\text{for all} \, \mathcal{M}'\in\mathbf{N} \; \text{ with } \; \mathcal{M}_1\leq \mathcal{M}'\leq \mathcal{M}_2.
\end{equation}
	\item[(A1)] \textit{Stabilization region:} There exists a map $R$ from $\{(x,\mathcal{M})\in\X\times \mathbf{N}:x\in\mathcal{M}\}$
	to $\cF$ such that
	\begin{enumerate}[(1)]
		\item[(A1.1)] 
		\begin{equation}
			\label{eq:1}
			\{\mu \in \mathbf{N} : y\in R(x,\mu+\delta_x)\}\in\mathscr{N} \quad \forall \, (x,y)\in\X^2
		\end{equation}
		 and
		\begin{equation}
			\label{eq:2}
			\P(y\in R(x, \eta+\delta_x))\, \text{ and } \,
			\P\left(\{y_1, y_2\}\subseteq  R(x, \eta+\delta_x)\right)
		\end{equation}
		are Lebesgue measurable functions of $(x,y) \in \X^2$ and
		$(x,y_1,y_2) \in \X^3$, respectively.
		\item[(A1.2)] the map $R$ is monotonically decreasing in the second argument,
		i.e.\
		$$
		R(x,\mathcal{M}_1) \supseteq R(x,\mathcal{M}_2),
		\quad \mathcal{M}_1 \leq \mathcal{M}_2,\; x\in\mathcal{M}_1,
		$$ 
		\item[(A1.3)] for all $\cM\in\mathbf{N}$ and $x\in\cM$, $\cM_{R(x,\cM)} \neq 0$
		implies $\cM_{R(x,\cM +\delta_y)} \neq 0$ for all
		$y \notin R(x,\cM)$,  
		\item[(A1.4)] for all $\mathcal{M}\in\mathbf{N}$ and $x\in\mathcal{M}$,
		\begin{displaymath}
			\xi\big(x,\mathcal{M}\big)
			=\xi\big(x,\mathcal{M}_{R(x,\mathcal{M})}\big).
		\end{displaymath}
	\end{enumerate}
\end{enumerate}
Note that (A1) holds trivially if one takes $R$ to be identically
equal to the whole space $\X$. If (A1) holds with a non-trivial
$R$, then the score function is called
\emph{region-stabilizing}.

\begin{enumerate}
	\item[(A2)] \textit{$L^{4+p}$-norm:} There exists a $p >0$
	such that, for all $x \in \X$ and $\mu\in\mathbf{N}$ with $\mu(\X) \le 7$,
	\begin{displaymath}
		\Big\|\xi\big(x, \eta+\delta_x+\mu\big)\Big\|_{4+p}
		\leq M_{p}<\infty.
	\end{displaymath}
\end{enumerate}
Further, let $r: \X \times \X \to [0,\infty]$ be a measurable function
such that for all $x, y \in \X$,
\begin{equation}
	\label{eq:Rs}
	\P(y \in R(x, \eta +\delta_x))
	\le  e^{-r(x,y)}.
\end{equation}
Note that we allow $r$ to be infinite. For
$x_1,x_2 \in \X$, denote
\begin{equation}
	\label{eq:g2s}
	q(x_1,x_2):=\int_\X \P\Big\{\{x_1,x_2\} 
	\subseteq R\big(z, \eta +\delta_z\big)\Big\} \;\mathbb{Q}(\md z).
\end{equation}
Denote $\zeta:=p/(40+10p)$ with $p >0$ as in (A2). For $y\in\X$, define the functions
\begin{gather}
	\label{eq:g}
	g(y) :=\int_{\X} e^{-\zeta r(x, y)} \;\mathbb{Q}(\md x), \text{ and }\\
	G(y) := \max(M_p^2,M_p^4)+\max(M_p^2g(y)^{2/(4+p/2)}, M_p^4g(y)^{4/(4+p/2)})(1+g(y)^4).
\end{gather}
For
$\alpha>0$, let
\begin{equation}
	\label{eq:fa}
	f_\alpha(y):=f_\alpha^{(1)}(y)+f_\alpha^{(2)}(y)+f_\alpha^{(3)}(y),
	\quad y\in\X,
\end{equation}
where 
\begin{align}
	\label{eq:fal}
	f_\alpha^{(1)}(y)&:=\int_\X G(x) e^{- \alpha r(x,y)}
	\;\mathbb{Q}(\md x), \notag \\
	f_\alpha^{(2)} (y)&:=\int_\X G(x) e^{- \alpha r(y,x)} 
	\;\mathbb{Q}(\md x), \nonumber\\
	f_\alpha^{(3)}(y)&:=\int_{\X} G(x) q(x,y)^\alpha \;\mathbb{Q}(\md x).
\end{align}
Finally, define the function
\begin{equation}
	\label{eq:p}
	\kappa(x):= \P(\xi(x, \eta+\delta_x) \neq 0),\quad
	x\in\X. 
\end{equation}

For an integrable function $f : \X \to \R$, denote
$\mathbb{Q} f:=\int_\X f(x) \mathbb{Q}(\md x)$.

\begin{theorem}[\cite{BM22}, Theorem 2.1]\label{thm:KolBd}
	Assume that $\xi$ satisfy conditions (A0) - (A2) and
	let $H$ be as in \eqref{eq:scoresum}. Then for $p$ as in (A2) and
	$\tau:=p /(32+4 p)$, there exists a constant $C \in (0,\infty)$ depending only on $p$ such that
    $$
    \Var(H) \le C \; \mathbb{Q} ((\kappa+g)^{2\tau}G).
    $$
    Moreover,
	\begin{align*}
		d_{W}\left(\frac{H-\E H}{\sqrt{\operatorname{Var} H}},  N\right) 
		&\leq C \Bigg[\frac{\sqrt{\mathbb{Q} f_\tau^2}}{\operatorname{Var} H}
		+\frac{\mathbb{Q} ((\kappa+g)^{2\tau}G)}{(\operatorname{Var} H)^{3/2}}\Bigg],
	\end{align*}
	and
	\begin{multline*}
		d_{K}\left(\frac{H-\E H}{\sqrt{\operatorname{Var} H}},
		N\right) 
		\leq C \Bigg[\frac{\sqrt{\mathbb{Q} f_\tau^2}
			+ \sqrt{\mathbb{Q} f_{2\tau}}}{\operatorname{Var} H}
		+\frac{\sqrt{\mathbb{Q} ((\kappa+g)^{2\tau}G)}}{\operatorname{Var} H}
		\\
		+\frac{\mathbb{Q} ((\kappa+g)^{2\tau} G)}{(\operatorname{Var} H)^{3/2}} +\frac{(\mathbb{Q} ((\kappa+g)^{2\tau} G))^{5/4}
			+ (\mathbb{Q} ((\kappa+g)^{2\tau} G))^{3/2}}{(\operatorname{Var} H)^{2}}\Bigg],
	\end{multline*}
	where $N$ is a standard normal random variable and
	$C \in (0,\infty)$ is a constant depending only on $p$.
\end{theorem} 

\begin{remark}[Applicability in Laguerre tessellations] 
    In the setting of Theorem \ref{thm:KolBd}, fixing $\X=\R^d \times \R$ with an associated measure $\Q(\md v, \md h)$ which we will take as the intensity measure of the processes $\eta_\beta, \eta'_\beta$ or $\widetilde \eta$, we note that $F_n(\eta)$ in Theorem \ref{thm:1} can indeed be expressed as a sum of score functions $\xi_n: \X \times \bN \to \R$ given by
$$
\xi_n((v,h), \cM):= \ind (v \in W_n) \ind ((v,h) \text{ is an extreme point in $\cM$}),
$$
for any locally finite point configuration $\cM \in \bN$ and $(v,h) \in \cM$. In Section \ref{sec:proof1.1}, we show that this is indeed region-stabilizing and it satisfies all the conditions in Theorem \ref{thm:KolBd}. Applying Theorem \ref{thm:KolBd} then yields presumably optimal (in the sense of Berry-Esseen bound) rates of the order of $(\operatorname{Var}(F_n(\eta)))^{-1/2}$ for the Gaussian convergence. We can prove such rates for $\widetilde \eta$, as well as for the full admissible range of $\beta>-1$ in the case of $\eta_\beta$. However, for the $\beta'$ model, while we can estimate all other terms in the bounds of Theorem \ref{thm:KolBd} appropriately assuming only $\beta>3d+1$, when estimating the integral $\Q f_\alpha^{(3)}$ (with $\alpha=\tau$ or $2\tau$), we are forced to assume $\beta>5d+1$ for certain integrals to be finite; in particular, the integral of the function in \eqref{eq:integ5d} is finite only when $\beta>5d+1$. We believe that it might be possible to relax this restriction, though it seems unlikely that we can obtain a result for all $\beta>d/2+1$ using our approach.
\end{remark}

In \cite[Theorem 2.1]{BM22}, the parameter $p$ above is taken to be in the interval $(0,1]$. But the arguments for its proof remain valid for any choice of $p>0$, and for the purposes of the current paper, we indeed take $p>0$. This becomes useful in the proof for the $\beta'$-model where we choose $p>1$ to optimize the range of $\beta$ where we are able to prove our result.

Also, we note here that the variance upper bound above is not stated explicitly in \cite[Theorem 2.1]{BM22}, though it follows from computations therein. For completeness, we provide a short argument below showing this.
\begin{proof}[Proof of Variance upper bound in Theorem \ref{thm:KolBd}]
    The variance upper bound follows from combining elements from the proof of Theorem 2.1 in \cite{BM22}. By the Poinc{\'a}re inequality (see e.g.\ \cite{lastpenrose}), we have that
$$
\Var(H) \le \int_\X \E (D_x H)^2\, \Q(\md x).
$$
Arguing as in the proof of \cite[Theorem 5.1]{BM22}, recalling that $\tau=p/(32+4p)$, H\"{o}lder's inequality yields that
\begin{equation}\label{eq:Poi}
\E (D_x H)^2 \le \left(\E |D_x H|^{4+p/2}\right)^{\frac{2}{4+p/2}} \P(D_x H \neq 0)^{\frac{2+p/2}{4+p/2}} \le \left(\E |D_x H|^{4+p/2}\right)^{\frac{2}{4+p/2}} \P(D_x H \neq 0)^{2\tau}.
\end{equation}
As in \cite[Lemma 5.7]{BM22}, one can now bound 
$$
\P(D_x H \neq 0) \le \kappa(x) + g(x),
$$
while by Lemma 5.5 therein (with a slight simplification due to our uniform $L^{4+p}$-norm bound assumption in (A2)),
$$
\left(\E |D_x H|^{4+p/2}\right)^{\frac{2}{4+p/2}} \le C_p G(x)
$$
with $C_p \in [1,\infty)$ depending
only on $p$. Thus, from \eqref{eq:Poi}, it follows that
\begin{align*}
    \Var(H) &\le C_p \Q ((\kappa + g)^{2\tau} G)
\end{align*}
yielding the upper bound.
\end{proof}

\section{Properties of Poisson-Laguerre tessellation}\label{sec:Laguerre}

In this section, we establish some fundamental results about the Laguerre tessellations considered here. In particular, in Proposition \ref{lm:tailbounds}, we provide tail estimates for the growth time of a typical cell in our three models, while Proposition \ref{prop:VarianceBound} provides variance lower bounds for $F_n(\eta)$ of the order of the volume of the window $W_n$ in all our models.

\subsection{Tail bounds for the coverage time}
Consider some locally finite counting measure $\cM \in \bN$, which we identify with its support, defined as in Section \ref{sec:cltBM} with $\X = \R^d \times \R$. Recall that given $x=(v,h) \in \cM$, its cell is defined as
\[
C(x,\cM):=\{w\in\R^d\colon \p(w,x)\leq \p(w,x')\quad\forall x'\in\cM\}.
\]
We may view the cell $C(x,\cM)$ as the result of a crystal growth process. More precisely, we associate with the point  $x=(v,h) \in \cM$ a crystal growth process, where a crystal starts its growth at point $v 
\in \R^d$ and at time $h \in \R$ with the same speed in all directions and the speed of growth slows down with time and is $1/2\sqrt{t}$. Then a point $w\in\R^d$ belongs to the cell of that point $x\in\cM$ whose crystal reaches $w$ first. The result of every crystal growth process is then a collection of crystals/cells $\{C(x,\cM)\}_{x \in \cM}$. With this interpretation, for any $t\ge h$ we may define the cell of $x$ at time $t$ as
\begin{align*}
C_t(x,\cM)&:=\{w\in B^d(v,\sqrt{t-h})\colon \p(w,x)\leq \p(w,x')\; \forall\;  x'\in\cM\}\\
&=C(x,\cM)\cap B^d(v,\sqrt{t-h}),
\end{align*}
which is the result of the crystal growth process at $x$ by time $t$. In case $t<h$ we set $C_t(x,\cM)=\emptyset$. Note that for any $t'>t$ we have $C_t(x,\cM)\subseteq C_{t'}(x,\cM)$, since the cells always grow, or stay the same with time. Further given $x=(v,h)\in \cM$, define the \textit{coverage time} (see also Figure \ref{Fig:RS})
\begin{equation}\label{eq:covertime1}
    T(x,\cM):= \inf_{t\in \R}\{C_t(x,\cM)=C(x,\cM)\}
\end{equation}
of the cell of a point $x \in \cM$, which is the last time when the cell of $x$ sees any growth. Recall the definition of an extreme point from the introduction, and note that if $x$ is not an extreme point, then $C(x,\cM)=\emptyset$ and, hence, $T(x,\cM)=-\infty$. If $x$ is an extreme point we have an alternative representation for $T(x,\cM)$.
\begin{lemma}\label{lm:covertime} For any extreme point $x=(v,h)\in\cM$, we have
\begin{align*}
    T(x, \cM)&= \sup_{t \ge h} \big\{\partial B^{d}(v,\sqrt{t-h}) \not \subseteq \cup_{x'\neq x\in\cM}C_t(x',\cM) \big\}.
\end{align*}
\end{lemma}

\begin{proof}
Indeed, if $t\ge T(x,\cM)\ge h$, then 
\[
 C(x,\cM)\cap B^d(v,\sqrt{t-h})=C_t(x,\cM)=C(x,\cM).
\]
According to the properties of the crystal growth process, this implies that for any $w\in \partial B^d(v,\sqrt{t-h})$ there is $x'\neq x\in\cM$ such that $w\in C_t(x',\cM)$. Hence, 
\[
\partial B^{d}(v,\sqrt{t-h})\subset \cup_{x'\in\cM}C_t(x',\cM)
\]
and
\begin{equation}\label{22.01.25_eq2}
T(x,\cM)\ge  \sup_{t \ge h} \big\{\partial B^{d}(v,\sqrt{t-h}) \not \subseteq \cup_{x'\neq x\in\cM}C_t(x',\cM) \big\}.
\end{equation}
At the same time if $t\leq T(x,\cM)$, then 
\begin{equation}\label{22.01.25_eq1}
C_t(x,\cM)= C(x,\cM)\cap B^d(v,\sqrt{t-h})\neq C(x,\cM).
\end{equation}
This implies, that either $ C_t(x,\cM)=\emptyset$ or there is $w\in\partial B^d(v,\sqrt{t-h})$, such that $w\in {\rm int}\, C(x,\cM)$. If $ C_t(x,\cM)=\emptyset$ since $x$ is extreme there is $t'\in (t,T(x,\cM)]$ such that  $ C_{t'}(x,\cM)\neq\emptyset$ and we may assume without loss of generality, that $ C_t(x,\cM)\neq \emptyset$. Now if for any $w\in\partial B^d(v,\sqrt{t-h})$ we have that $w\not\in  {\rm int}\, C(x,\cM)$, then since $C(x,\cM)$ is connected we have $C(x,\cM)\subset B^d(v,\sqrt{t-h})$, which contradicts \eqref{22.01.25_eq1}. Then according to the rules of crystal growth, there is no $x'\neq x\in \cM$ such that $w\in C(x',\cM)$ and, hence, $w\not\in \cup_{x'\neq x\in\cM}C_t(x',\cM)$ implying $\partial B^{d}(v,\sqrt{t-h}) \not \subseteq \cup_{x'\neq x\in\cM}C_t(x',\cM)$. Moreover the same holds for any $t<t'<T(x,\cM)$ and thus,
\[
t< \sup_{t \ge h} \big\{\partial B^{d}(v,\sqrt{t-h}) \not \subseteq \cup_{x'\neq x\in\cM}C_t(x',\cM) \big\},
\]
implying
\[
T(x,\cM)\leq  \sup_{t \ge h} \big\{\partial B^{d}(v,\sqrt{t-h}) \not \subseteq \cup_{x'\neq x\in\cM}C_t(x',\cM) \big\}.
\]
Together with \eqref{22.01.25_eq2} this completes the proof.
\end{proof}

Let $x=(v,h)\in \R^d\times \R$ be some fixed point. In what follows we will be using the following tail bounds for $T(x, \eta \cup \{x\})$, where $\eta \in \{\eta_\beta, \eta'_\beta, \widetilde \eta\}$.

\begin{prop}\label{lm:tailbounds}
    Let $x=(v,h)\in \R^d\times \R$ be some fixed point and $H\in\R$. Then 
\begin{align*}
    \P(T(x,\eta_{\beta}\cup\{x\})> H) &\leq  C_1\exp\big(-\gamma c_1\big(((H\vee h)-1)\vee 0\big)^{d/2+\beta+1}\big),\quad &\forall\, H,h\ge 0,\\
    \P(T(x,\eta'_{\beta}\cup\{x\})> H)
    	&\leq  C_2\exp\big(-\gamma c_2(-(H\vee h))^{-(\beta-d/2-1)}\big),\quad &\forall\, H,h<0,\\
    \P(T(x,\widetilde \eta\cup\{x\})> H)
    	&\leq  C_3\exp\big(-c_3e^{(H\vee h)/2}\big),\quad &\forall\, H\in \R, h\ge 0 \text{ or } H\leq 0, h< 0,\\
        \P(T(x,\widetilde \eta\cup\{x\})> H)
    	&\leq  C_3(1+|h|)^{d/2}\exp\big(-c_3e^{(H\vee h)/2}\big),\quad &\forall\, H>0, h<0,
\end{align*}
where $c_1, c_2, C_1\in (0,\infty)$ are some constants depending on $d$ and $\beta$ only and $c_3, C_2, C_3\in (0,\infty)$ are some constants depending on $d$ only.
\end{prop}

\begin{remark}
    Keeping track of the constants along the lines of the proof of Proposition \ref{lm:tailbounds} one may ensure, that $c_1,C_1\in (0,\infty)$ also for $\beta=-1$. At the same time $c_2$ can be chosen to be
    \[
    c_2={\Gamma(\beta-{d\over 2}-1)\over 4^{2d+3-2\beta}\Gamma(\beta-{d+1\over 2})}
    \]
    which belongs to $(0,\infty)$ for any $\beta>d/2+1$, but goes to infinity as $\beta$ approaches it's critical value $d/2+1$.
\end{remark}

\begin{proof}
    Along this proof we will use notation $c, C\in (0,\infty)$ for some constants, which may only depend on $d$ and possibly $\beta$ and whose value may change from line to line.
    
    Let $\eta$ be one of the processes $\eta_{\beta}$, $\eta_{\beta}'$ or $\widetilde \eta$ and consider $T=T(x,\eta\cup\{x\})$. Further, since all three processes $\eta_{\beta}$, $\eta'_{\beta}$ and $\widetilde \eta$ are stationary with respect to their spacial coordinate, we may without loss of generality assume that $v=0$. Note that if $h> H$, then 
    \begin{align*}
    \{T\leq H\} &= \{T=-\infty\}\\
    &=\{C(x,\eta\cup\{x\})=\emptyset\}\\
    &=\big\{\partial B^d(0,\sqrt{t-h})\subseteq \bigcup_{x'\in \eta}C_t(x',\eta\cup\{x\})\text{ for all }t\ge h \big\}.
    \end{align*}
    On the other hand, when $h\leq H$, by Lemma \ref{lm:covertime} we have
    \[
    \{T\leq H\}=\big\{\partial B^d(0,\sqrt{t-h})\subseteq \bigcup_{x'\in\eta}C_t(x',\eta\cup\{x\})\text{ for all }t\ge H\big\}.
    \]
    Combining, we conclude that the event $\{T\leq H\}$ is equivalent to the event
     \[
    \cE:=\big\{\partial B^d(0,\sqrt{t-h})\subseteq \bigcup_{x'\in\eta}C_t(x',\eta\cup\{x\})\text{ for all }t\ge H\vee h\big\}
    \]
    whence
    \begin{align*}
    \P(T> H)&=\P(\cE^c)\\
    &=\P\Big(\exists t\ge H\vee h\colon \partial B^d(0,\sqrt{t-h})\not\subseteq \bigcup_{x'\in\eta}C_t(x',\eta\cup\{x\}) \Big)\\
    &= \P\Big(\exists t\ge H\vee h, w\in \partial B^d(0,\sqrt{t-h}) \colon w\not\in \bigcup_{x'\in\eta}C_t(x',\eta\cup\{x\}) \Big).
    \end{align*}
    Further, note that if $w\in \partial B^d(0,\sqrt{t-h})$ but $w\not\in \bigcup_{x'\in\eta}C_t(x',\eta\cup\{x\})$, this implies that $w\not \in \bigcup_{x'=(v',h')\in\eta, h'\leq t}B^d(v',\sqrt{t-h'})$. Indeed, assume there are $(v_1,h_1)$, $\ldots$, $(v_k,h_k)\in\eta$ such that $h_i\leq t$, $w\in B^d(v_i,\sqrt{t-h_i})$ for all $1\leq i\leq k$, $w\in \partial B^d(0,\sqrt{t-h})$ and $w\not\in  B^d(v',\sqrt{t-h'})$ for any other point $(v',h')\in\eta$. This implies 
    \begin{align*}
    \p(w,(v_i,h_i))&\leq t,\qquad 1\leq i\leq k,\\
    \p(w,(0,h))&=t,\\
    \p(w,(v',h'))&>t,\qquad (v',h')\in \eta\setminus \{(v_i,h_i),1\leq i\leq k\}.
    \end{align*}
    This means that $w\in C_t((v_i,h_i),\eta)$
    for some $1\leq i\leq k$, which is a contradiction.

    Hence, denoting by 
    \[
    \Pi_+(v',h'):=\{(w,s)\in\R^d\times \R\colon \|w-v'\|^2+h'= s\}
    \]
    the upward  paraboloid with apex $(v',h')$ and by 
    \[
    \Pi^{\downarrow}_-(v',h'):=\{(w,s)\in\R^d\times \R\colon -\|w-v'\|^2+h'\ge s\}
    \]
    the set of points lying below the downward paraboloid with apex $(v',h')$, we obtain
    \begin{align}
        \P(T> H)&\leq \P\Big(\exists t\ge H\vee h, w\in \partial B^d(0,\sqrt{t-h}) \colon w\not\in \bigcup_{(v',h')\in\eta,h'\leq t}B^d(v',\sqrt{t-h'}) \Big)\notag\\
        &=\P\big(\exists (w,t)\in \Pi_+(0,h)\colon t\ge H,\, t-\|w-v'\|^2<h'\quad \forall (v',h')\in \eta\big)\notag\\
        &=\P\big(\exists (w,t)\in \Pi_+(0,h)\colon t\ge H, \Pi^{\downarrow}_-(w,t)\cap \eta=\emptyset\big).\label{eq:24.01.25_1}
    \end{align}

    In order to estimate the above probability, we will discretise the space and cover the paraboloid surface $\Pi_+(0,h)$ with small cylinders $Z_i$. Then we will control the probability that there is $ (w,t)\in Z_i$ such that $\Pi^{\downarrow}_-(w,t)\cap \eta=\emptyset$ and use the union bound.

    To that end, consider a cylinder $Z:=B^d(u, R)\times [a,b]$, $u\in\R^d$, $R>0$, $a<b$. Then 
    \begin{align}
        P(Z,\eta)&:=\P\big(\exists (w,t)\in Z\colon \Pi^{\downarrow}_-(w,t)\cap \eta=\emptyset\big)\notag\\
        &\leq \P\big( K(R,a,u)\cap \eta=\emptyset\big)\notag\\
        &=\exp\Big(-\E \,|K(R,a,u)\cap \eta|\Big),\label{eq:24.01.25_2}
    \end{align}
    where 
    \[
    K(R,a,u):=\{(v',h')\in\R^d\times\R\colon h'\leq a-R^2,\|v'-u\|\leq \sqrt{a-h'}-R\}.
    \]
    Indeed, it is easy to verify that for any $(w,t)\in Z$, we have that $K(R,a,u)\subset \Pi^{\downarrow}_-(w,t)$, so that $\Pi^{\downarrow}_-(w,t)\cap \eta=\emptyset$ implies $K(R,a,u)\cap \eta=\emptyset$. 
    
    In the next step, we estimate this expectation for our three models. First, let $\eta=\eta_{\beta}$ and fix $a\ge 4 R^2$. Then we have
    \begin{align*}
        \E\,|K(R,a,u)\cap \eta_{\beta}| &=\gamma c_{d,\beta}\int_{0}^{a-R^2}\int_{\R^d}\ind\{\|v-u\|\leq \sqrt{a-h}-R\}\md v h^{\beta}\md h\\
        &=\gamma \kappa_d c_{d,\beta}\int_{0}^{a-R^2}(\sqrt{a-h}-R)^dh^{\beta}\md h\\
        &\ge \gamma \kappa_d c_{d,\beta}\int_{0}^{a-4R^2}(\sqrt{a-h}-R)^dh^{\beta}\md h.
    \end{align*}
    Now for $h\leq a-4R^2$, note that $\sqrt{a-h}\ge 2R$ so that $\sqrt{a-h}-R\ge {1\over 2} \sqrt{a-h}\ge {1\over 2} \sqrt{a-4 R^2-h}$. Together with the above inequality, this implies that for any $\beta\ge -1$ we have
    \begin{align*}
    \E \,|K(R,a,u)\cap \eta_{\beta}| &\ge  \gamma 2^{-d}\kappa_d c_{d,\beta}\int_{0}^{a-4 R^2}(a-4 R^2-h)^{d/2}h^{\beta}\md h\\
    &\ge \gamma {\Gamma({d\over 2}+\beta+{3\over 2})\over 2^{d+2}\Gamma({d\over 2}+\beta+2)} (a-4 R^2)^{d/2+\beta+1},\qquad a\ge 4 R^2.
    \end{align*}
    Combining this with \eqref{eq:24.01.25_2} in the case $a\ge 4 R^2$, and using the trivial bound otherwise, we obtain
    \begin{equation}\label{eq:24.01.25_beta}
        P(Z,\eta_{\beta})\leq\begin{cases}
        \exp\big(-\gamma c\,(a-4 R^2)^{d/2+\beta+1}\big),\,& a\ge 4 R^2,\\
        1,\, &a< 4 R^2.
        \end{cases}
    \end{equation}

    Mext let $\eta=\eta'_{\beta}$. For $a\ge 4 R^2$ we have $ \E(K(R,a,u)\cap \eta'_{\beta})=\infty$ and  hence, $P(Z,\eta_{\beta}')=0$. Indeed, if in this case we have
    \begin{align*}
        \E\,|K(R,a,u)\cap \eta'_{\beta}|&=\gamma \kappa_d c'_{d,\beta}\int_{-\infty}^{0}(\sqrt{a-h}-R)^d (-h)^{-\beta}\md h\\
        &\ge \gamma \kappa_d c'_{d,\beta}R^d\int_{-\infty}^{0}(-h)^{-\beta}\md h=\infty.
    \end{align*}
    If $a<4R^2$, then similar to the previous case, we obtain
    \begin{align*}
        \E\,|K(R,a,u)\cap \eta'_{\beta}|&\ge \gamma 2^{-d}\kappa_d c'_{d,\beta}\int_{-\infty}^{a-4R^2}(a-4R^2-h)^{d/2}(-h)^{-\beta}\md h\\
    &= \gamma {\Gamma(\beta-{d\over 2}-1)\over 2^{d+2}\Gamma(\beta-{d+1\over 2})} (4 R^2-a)^{d/2+1-\beta},\qquad a<R^2.
    \end{align*}
    Combining this with \eqref{eq:24.01.25_2} we get 
    \begin{equation}\label{eq:24.01.25_beta_prime}
        P(Z,\eta'_{\beta})\leq\begin{cases}
        \exp\Big(- \gamma c\, (4 R^2-a)^{-(\beta-d/2-1)}\Big),\,& a<4 R^2,\\
        0,\, &a\ge 4 R^2.
        \end{cases}
    \end{equation}

    Finally, for $\eta=\widetilde \eta$ we obtain
     \begin{align*}
    \E\,|K(R,a,u)\cap \widetilde\eta|&\ge  2^{-d}\kappa_d\int_{-\infty}^{a-4 R^2}(a-4 R^2-h)^{d/2}e^{h}\md h= \big(\pi/2\big)^de^{a-4 R^2},
    \end{align*}
    and together with \eqref{eq:24.01.25_2} this implies 
    \begin{equation}\label{eq:24.01.25_Gaussian}
        P(Z,\widetilde\eta)\leq\exp\big(-  c\,e^{a-4 R^2}\big).
    \end{equation}

    As a next step we introduce the covering of $\Pi_+(0,h)$. For simplicity, we write $\tilde t:=H\vee h$. For $y>0$ a parameter to be chosen later and $m\in\N_0$, consider $u\in \ell(m)\mathbb{Z}^d$, where 
    \[
    \ell(m)=\sqrt{{\widetilde t+y(m+1)-(h\wedge 0)\over 2d}}\in \R_+,
    \]
    which is well defined since $\widetilde t=H\vee h\ge h\ge (h\wedge 0)$. Then we define the cylinders
    \[
    Z(m, u):=B^d\Big(u, {\sqrt{d}\ell(m)\over 2}\Big)\times [\tilde t+ym, \tilde t+y(m+1)].
    \]
    By this definition, we clearly have that
    \[
    \{(w,t)\in\Pi_+(0,h)\colon t\ge H\}\subset \bigcup_{m,u}Z(m,u).
    \]
    Hence by \eqref{eq:24.01.25_1} and the union bound, we have
    \begin{equation}\label{eq:24.01.25_3}
        \P(T>H)\leq \sum_{m=0}^{\infty}\sum_{u\in \ell(m)\mathbb{Z}^d}P(Z(m,u),\eta)\ind\{Z(m,u)\cap \Pi_+(0,h)\neq \emptyset\}.
    \end{equation}
    Note that the probability $P(Z(m,u),\eta)$ is independent of $u$ and can be bounded as in \eqref{eq:24.01.25_beta}, \eqref{eq:24.01.25_beta_prime} and \eqref{eq:24.01.25_Gaussian}. It remains to estimate the number $N(m)$ of cylinders $Z(m,u)$ hitting $\Pi_+(0,h)$ for a given $m\in\N_0$. This can be done as follows. Denote by $A(r_1, r_2):=\{v\in\R^d\colon r_1\leq \|v\|\leq r_2\}$ the annulus with inner radius $r_1$ and outer radius $r_2>r_1$. Then for $m\ge 0$ we have
    \begin{align}
    N(m)&= \#\big\{u\in \ell(m)\mathbb{Z}^d\colon Z(m,u)\cap \Pi_+(0,h)\neq \emptyset\big\}\notag\\
    &= \#\left\{u\in \ell(m)\mathbb{Z}^d\colon B^d\left(u, {\sqrt{d}\ell(m)\over 2}\right)\cap A\left(\sqrt{\widetilde t-h+ym}, \sqrt{\widetilde t-h+y(m+1)}\right)\neq \emptyset\right\}\notag\\
    &\leq \#\left\{u\in \ell(m)\mathbb{Z}^d\colon u\in B^d\left(0,\sqrt{\widetilde t-h+y(m+1)}+{\sqrt{d}\ell(m)\over 2}\right)\right\}\notag\\
    &\leq \ell(m)^{-d}{\rm vol}\, \left(B^d\left(0,\sqrt{\widetilde t-h+y(m+1)}+\sqrt{d}\ell(m)\right)\right)\notag\\
    &= \kappa_d\left(\sqrt{d}+\sqrt{2d(\widetilde t+y(m+1)-h)\over \widetilde t+y(m+1)-(h\wedge 0)}\right)^{d}\leq \kappa_d(3\sqrt{d})^d.\label{eq:24.01.25_4}
    \end{align}

    For $\eta=\eta_{\beta}$ we set $y=1$ and consider only $h\ge 0$, $H\ge 0$. Now we combine \eqref{eq:24.01.25_3} with \eqref{eq:24.01.25_4} and apply \eqref{eq:24.01.25_beta} with
    \[
    a=\widetilde t+m,\qquad R={\sqrt{\widetilde t+m+1}\over 2\sqrt{2}}.
    \]
    We note that for $m\ge 1-(h\vee H)$ we get $a\ge 4R^2$. 
    Thus, we obtain that
    \begin{align*}
        \P(T>H)&\leq C\Big[\ind((h\vee H)<1)+\ind((h\vee H)\ge 1)\exp\Big(-\gamma c\big((h\vee H)-1\big)^{d/2+\beta+1}\Big)\\
        &\qquad\qquad+\sum_{m=1}^{\infty}\exp\Big(-\gamma c\,\big((h\vee H)+m-1\big)^{d/2+\beta+1}\Big)\Big]\\
        &\leq C\Big[\exp\Big(-\gamma c\,\big(((h\vee H)-1)\vee 0\big)^{d/2+\beta+1}\Big)+\sum_{m=0}^{\infty}\exp\Big(-\gamma c\,\big((h\vee H)+m\big)^{d/2+\beta+1}\Big)\Big]\\
        &\leq C\Big[\exp\Big(-\gamma c\,\big(((h\vee H)-1)\vee 0\big)^{d/2+\beta+1}\Big)+\exp\Big(-\gamma c\,(h\vee H)^{d/2+\beta+1}\Big)\sum_{m=0}^{\infty}e^{-\gamma c\,m}\Big]\\
        &\leq C\Big[\exp\Big(-\gamma c\,\big(((h\vee H)-1)\vee 0\big)^{d/2+\beta+1}\Big)+\big(1-e^{-\gamma c}\big)^{-1}\exp\Big(-\gamma c\,(h\vee H)^{d/2+\beta+1}\Big)\Big],
    \end{align*}
    where in the second step we used that $(a+b)^{\alpha}\ge a^{\alpha}+b^{\alpha}$ for any $a,b\ge 0$ and $\alpha\ge 1$ together with the fact that $d/2+\beta+1\ge 1$ for any $\beta>-1$ and $d\ge 2$. This finishes the proof of the first bound.

    In the case $\eta=\eta'_{\beta}$, we set $y=-(\widetilde t+h)/2>0$. Further, we combine \eqref{eq:24.01.25_3} with \eqref{eq:24.01.25_4} and \eqref{eq:24.01.25_beta_prime} choosing
    \[
    a=\widetilde t-{m(\widetilde t+h)\over 2},\qquad R={\sqrt{2\widetilde t-(m+1)(\widetilde t+h)-2(h\wedge 0)}\over 4}.
    \]
    Note that the condition $a<4R^2$ implies that $m\leq 2$. Thus we obtain
    \begin{align*}
        \P(T>H)&\leq C\sum_{m=0}^{2}\exp\Big(-\gamma c\,\big(-((H\vee h)+h)(m+1)\big)^{-(\beta-d/2-1)}\Big)\\
        &\leq C\exp\big(-\gamma c\,(-(H\vee h))^{-(\beta-d/2-1)}\big)
    \end{align*}
    yielding the second assertion.
    
    Finally we consider case when $\eta=\widetilde\eta$. We start by deriving a tail bound for any $H, h\in \R$. In the case when $H>0$ and $h<0$, this general bound however appears to be not good enough for our purposes, and hence we treat this case separately at the end. 
    
    Set $y=1$. Combining \eqref{eq:24.01.25_3} with \eqref{eq:24.01.25_4} and \eqref{eq:24.01.25_Gaussian}, and choosing
    \[
    a=\widetilde t+m,\qquad R={\sqrt{\widetilde t+m+1-(h\wedge 0)}\over 2\sqrt{2}},
    \]
     we obtain
    \begin{align*}
         \P(T>H)&\leq C\sum_{m=0}^{\infty}\exp\big(-c\,e^{{\widetilde t\over 2}+{m-1\over 2}+{(h\wedge 0)\over 2}}\big)\\
        &\leq C\exp\big(-c\,e^{((H\vee h)+(h\wedge 0))/2}\big)\sum_{m=0}^{\infty}e^{-cm/2}\\
        &\leq C\exp\big(-c\,e^{((H\vee h)+(h\wedge 0))/2}\big),
    \end{align*}
    where in the second inequality, we have used $e^{y}\ge y+1$. Further, we note that if $H\leq 0$ and $h\leq 0$ then
    \[
    \exp\big(-c\,e^{((H\vee h)+(h\wedge 0))/2}\big)\leq 1,
    \]
    while 
    \[
    \exp\big(-c\,e^{(H\vee h)/2}\big)\ge e^{-c}.
    \]
    Hence, in this case it holds that 
    \[
    \P(T>H)\leq Ce^c\exp\big(-c\,e^{(H\vee h)/2}\big).
    \]
    If $h\ge 0$, then one trivially has
    \[
    \P(T>H)\leq C\exp\big(-c\,e^{(H\vee h)/2}\big).
    \]

    Now assume $H>0$ and $h<0$. Then we start with a different mesh size and let
    \[
    \ell(m)=\sqrt{{\widetilde t+m+1\over 2d}}\in \R_+,
    \]
    which is well defined since $\widetilde t=H\vee h\ge H>0$. We consider the cylinders 
    \[
    Z(m, u):=B^d\Big(u, {\sqrt{d}\ell(m)\over 2}\Big)\times [\tilde t+m, \tilde t+m+1].
    \]
    Arguing similarly as for the bound \eqref{eq:24.01.25_4} and noting that $\widetilde t+m +1 \ge 1$, we obtain
        \begin{align}
    N(m)&\leq \kappa_d\Big(\sqrt{d}+\sqrt{2d -{2dh\over \widetilde t+m+1}}\Big)^{d}\leq \kappa_d(4\sqrt{d})^d\big(1+|h|\big)^{d/2}.\label{eq:24.01.25_4_new}
    \end{align}
    Further, using the bound \eqref{eq:24.01.25_3} and  \eqref{eq:24.01.25_Gaussian} with 
    \[
    a=\widetilde t+m,\qquad R={\sqrt{\widetilde t+m+1}\over 2\sqrt{2}},
    \]
    we arrive at
    \begin{align*}
        \P(T>H)&\leq  C\big(1+|h|\big)^{d/2}\sum_{m=0}^{\infty}\exp\big(-c\,e^{{\widetilde t\over 2}+{m-1\over 2}}\big)\leq C\big(1+|h|\big)^{d/2}\exp\big(-c\,e^{((H\vee h)/2}\big),
    \end{align*}
    which concludes the proof.
\end{proof}

\subsection{Variance lower bounds}
As a final ingredient for proving our quantitative CLT in Theorem \ref{thm:1}, we require appropriate variance lower bounds for our statistic $F_n(\eta)$ which we provide in the next proposition.
\begin{prop}\label{prop:VarianceBound}
We have
\begin{align*}
    \Var(F_n(\eta_{\beta})) &\ge C_1\,n^d,\qquad \beta>-1,\\
    \Var(F_n(\eta'_{\beta})) &\ge C_2\,n^d,\qquad \beta>d/2+1,\\
    \Var(F_n(\widetilde \eta)) &\ge C_3\,n^d
\end{align*}
for constants $C_1, C_2, C_3\in (0,\infty)$ depending only on $d$ and possibly on $\beta$ and $\gamma$ (in $\beta$ and $\beta'$ models).
\end{prop}

\begin{remark}\label{rem:Variance}
    We note that $C_1$ is a constant between $0$ and $\infty$ for all values $\beta\ge -1$, while $C_2$ tends to $0$ as $\beta$ approaches its critical value $d/2+1$. 
    
    Let us consider the $\beta'$ model more closely. By combining \cite[Proposition 3]{GKT22} with \cite[Theorem 6]{GKT22} and \cite[Theorem 2]{GKT22}, we obtain that
    \begin{align*}
    \E(F_n(\eta'_{\beta}))&=(2n)^d\Big({ \gamma\Gamma(\beta-{d\over 2}-1)\over \sqrt{\pi}\Gamma(\beta -{d+1\over 2})}\Big)^{{d\over d-2\beta+2}}{\Gamma({(d+1)(2\beta-d-1)+1\over 2})\over\Gamma({(d+1)(2\beta-d-1)-d+1\over 2})}{\Gamma({(d+1)(2\beta-d-2)\over 2})\over\Gamma({(d+1)(2\beta-d-1)\over 2})}\\
&\qquad\times{\Gamma(d+1-{d\over d+2-2\beta})\over \Gamma({d+2\over 2})}{\Gamma(\beta -{d+1\over 2})^{d+1}\over \Gamma(\beta -{d\over 2}-1)^{d+1}}\mathbb{J}'_{d+1,1}\big(\beta-1/2),
    \end{align*}
    where $\mathbb{J}'_{d+1,1}\big(\beta-1/2)$ is the expected sum of internal angles of $d$-dimensional $\beta'$-simplex with parameter $\beta-1/2$ (see \cite{Kab21} for more details). We note that 
    \[
    \mathbb{J}'_{d+1,1}\big({d+1\over 2}\big)=\lim_{\beta\to d/2+1}\mathbb{J}'_{d+1,1}\big(\beta-1/2)\in (0,1),
    \]
    since in this case $\mathbb{J}'_{d+1,1}\big({d+1\over 2}\big)$ describes the expected angle sum of a $d$-dimensional random simplex, whose vertices are independent Cauchy distributed random points in $\R^d$, which in turn implies that the simplex is non-degenerate almost surely. Denote by $\varepsilon = \beta-{d\over 2}-1$. Then using that $\Gamma(z)={1\over z} + O(1)$ as $z\to 0$, we get
    \begin{align*}
        \lim_{\beta\to d/2+1}\E(F_n(\eta'_{\beta}))&=(2n)^d\,C\,\lim_{\varepsilon\to 0}\Big({\sqrt{\pi}\Gamma(\varepsilon+{1\over 2}) \over \gamma\Gamma(\varepsilon)}\Big)^{{d\over 2\varepsilon}}{\Gamma((d+1)\varepsilon)\Gamma(d+1+{d\over 2\varepsilon})\over \Gamma(\varepsilon)^{d+1}}\\
        &=(2n)^d\,C\,\lim_{\varepsilon\to 0}\Big({\pi \over \gamma\Gamma(\varepsilon)}\Big)^{{d\over 2\varepsilon}}\Gamma\Big(d+1+{d\over 2\varepsilon}\Big)=:(2n)^d\,C\,\lim_{\varepsilon\to 0} g(\varepsilon).
    \end{align*}
    Consider
    \begin{align*}
        \log g(\varepsilon)={d\over 2\varepsilon}\Big(\log \Big({\pi\over \gamma}\Big)-\log\Gamma(\varepsilon)\Big)+\log\Gamma\Big(d+1+{d\over 2\varepsilon}\Big).
    \end{align*}
    Using the Stirling approximation \cite[Equation 5.11.8]{NIST} $\log \Gamma(z+a)=\big(z+a-{1\over 2}\big)\log z-z+O(1)$ as $z\to\infty$ and $\log\Gamma(z)=-\log(z)+O(z)$ as $z\to 0$ then yields
    \begin{align*}
        \log g(\varepsilon)={d\over 2\varepsilon}\log \Big({d\pi\over 2e\gamma}\Big)+\Big(d+{1\over 2}\Big)\log\Big({d\over 2\varepsilon}\Big)+O(1).
    \end{align*}
    Hence, we obtain
    \[
    \lim_{\beta\to d/2+1}\E(F_n(\eta'_{\beta}))=
    \begin{cases}
    \infty,\qquad &\gamma\leq {d\pi\over 2e},\\
    0,\qquad &\gamma >{d\pi\over 2e}.
    \end{cases}
    \]
    This in particular implies that for $\gamma >{d\pi\over 2e}$ we have $\lim_{\beta\to d/2+1}\Var(F_n(\eta'_{\beta}))=0$.
\end{remark}

\begin{proof}[Proof of Proposition \ref{prop:VarianceBound}]
    We note here that in case of $\eta_{\beta}$ and $\widetilde \eta$, these bounds already appear in \cite[Theorem 7]{GKT23} and the bound in the case of $\eta'_{\beta}$ can be obtained following similar arguments. However for completeness and the readers convenience, we present a complete and unified proof here.

    Let $\eta$ be one the processes $\eta_{\beta}$, $\eta'_{\beta}$ or $\widetilde \eta$ and let $\Q$ be the corresponding intensity measure. The idea of the proof is to condition on a certain 'good event' $\mathcal{G}$ on which the variance would be strictly positive and then to use the following formula for the total variance:
    \begin{equation}\label{eq:TotalVariance}
    \Var(F_n(\eta))=\E(\Var(F_n(\eta)|\mathcal{G}))+\Var(\E(F_n(\eta)|\mathcal{G}))\ge \E(\Var(F_n(\eta)|\mathcal{G})).
    \end{equation}

    Let us start by constructing this `good' event. Let $\ell(\eta)\in \R$ (we will choose $\ell(\eta_\beta)=\ell(\widetilde \eta)=1/8$ and $\ell(\eta'_{\beta})=-1/16$). Let $\delta>0$ be sufficiently small (i.e.\ $\delta\leq 1/16$). For any $z\in\mathbb{Z}^d\cap [-n,n]^d$, consider the set 
    \[
    K(z):=\{(v,h)\in \R^{d+1}\colon h\leq \ell(\eta),\,\|v-z\|\leq 4\delta+\sqrt{\ell(\eta)-h}\}.
    \]
    Further, let $x_1(z),\ldots,x_{d+1}(z)\in \partial B^d(z,3\delta)$ be vertices of regular simplex and define 
    \[
    K_i(z)=B^d(x_i(z),\delta)\times [\ell(\eta)-19\delta^2,\ell(\eta)-18\delta^2]\subset K(z)
    \]
    for $1\leq i\leq d+1$. Finally let 
    \[
    K_0(z)=B^d(z,\delta)\times [\ell(\eta)-9\delta^2,\ell(\eta)]\subset K(z).
    \]
    Let $A(z)$ be the event that there is exactly one point of $\eta$ in each of the sets $K_0(z),\ldots, K_{d+1}(z)$ and no other points of the process in $K(z)$. Since all sets $K_0(z),\ldots, K_{d+1}(z)$ are disjoint for sufficiently small $\delta$ and due to independence properties of a Poisson point process, we get
    \[
        \P(A(z))=\P\Big(\big|\big(K(z)\setminus\bigcup_{i=0}^{d+1}K_i(z)\big)\cap \eta\big|=0\Big)\prod_{i=0}^{d+1}\P(|K_i(z)\cap \eta|=1)=\prod_{i=0}^{d+1}\Q(K_i(z))\exp(-\Q(K(z))),
    \]
    since for any Borel set $B\subset\R^{d+1}$ the random variable $|B\cap \eta|$ has the same distribution as Poisson random variable with parameter $\Q(B)$. Next we note that $\ell(\eta_{\beta})>\ell(\eta_{\beta})-18\delta^2>0$ for sufficiently small $\delta>0$ and, hence, considering the $\beta$-model,
    \begin{align*}
    \Q_{\beta}(K_i(z))&={\gamma\, c_{d,\beta}\kappa_d\delta^d\over \beta+1}\Big(\big(\ell(\eta_{\beta})-18\delta^2)^{\beta+1}-\big(\ell(\eta_{\beta})-19\delta^2\big)^{\beta+1}\Big)
    \in (0,\infty),\qquad 1\leq i\leq d+1,\\
    \Q_{\beta}(K_0(z))&={\gamma\, c_{d,\beta}\kappa_d\delta^d\over(\beta+1)}\Big(\ell(\eta_{\beta})^{\beta+1}-\big(\ell(\eta_{\beta})-9\delta^2\big)^{\beta+1}\Big)
    \in (0,\infty),\\
    \Q_{\beta}(K(z))&=\gamma\,c_{d,\beta}\kappa_d\int_{0}^{\ell(\eta_{\beta})}(4\delta+\sqrt{\ell(\eta_{\beta})-h})^dh^{\beta}\md h\in (0,\infty).
    \end{align*}
    Combining these together, we conclude that for any $z \in \mathbb{Z}^d\cap [-n,n]^d$ it holds that $\P(A(z))>C$, where $C>0$ is some constant depending only on $d$, $\delta$, $\gamma$ and $\beta$. Similar computations for $\eta'_{\beta}$, together with the fact that $\ell(\eta_{\beta}')<0$ lead to
\begin{align*}
\Q'_{\beta}(K_i(z))&={\gamma\,c'_{d,\beta}\kappa_d\delta^d\over\beta-1}\Big(\big(19\delta^2-\ell(\eta'_{\beta}))^{-\beta+1}-\big(18\delta^2-\ell(\eta'_{\beta})\big)^{-\beta+1}\Big)\in (0,\infty),\qquad 1\leq i\leq d+1,\\
\Q'_{\beta}(K_0(z))&={\gamma\,c'_{d,\beta}\kappa_d\delta^d\over \beta-1}\Big(\big(9\delta^2-\ell(\eta'_{\beta})\big)^{-\beta+1}-(-\ell(\eta'_{\beta}))^{-\beta+1}\Big)\in (0,\infty),\\
\Q'_{\beta}(K(z))&=\gamma\,c'_{d,\beta}\int_{-\infty}^{\ell(\eta'_{\beta})}(4\delta+\sqrt{\ell(\eta'_{\beta})-h})^d(-h)^{-\beta}\md h\in (0,\infty),
\end{align*}
and, hence, the estimate $\P(A(z))>C>0$ with constant $C$ depending only on $d$, $\delta$, $\gamma$ and $\beta$ holds also in this case. Finally, for $\widetilde \eta$, we obtain
\begin{align*}
\widetilde\Q(K_i(z))&=\kappa_d\delta^de^{\ell(\widetilde\eta)-18\delta^2}(1-e^{-\delta^2})\in (0,\infty),\qquad 1\leq i\leq d+1,\\
\widetilde\Q(K_0(z))&=\kappa_d\delta^de^{\ell(\widetilde\eta)}(1-e^{-9\delta^2})\in (0,\infty),\\
\widetilde\Q(K(z))&=\int_{-\infty}^{\ell(\widetilde\eta)}(4\delta+\sqrt{\ell(\widetilde\eta)-h})^de^h\md h\in (0,\infty),
\end{align*}
yielding $\P(A(z))>C>0$, where $C$ is a constant depending only on $d$.

Now let $\mathcal{G}$ be the $\sigma$-algebra that determines the configuration of $\eta$ outside of all those sets $K(z)$, $z\in\mathbb{Z}^d\cap[-n,n]^d$, such that $\ind(A(z))=1$. Consider $\E(\Var(F_n(\eta)|\mathcal{G}))$. Let $K(z)$ be one of the sets such that $\ind(A(z))=1$. Let $(V_i(z),H_i(z))\in K_i(z)$ be the unique point of the process in $K_i(z)$. According to the construction of $K(z)$ and the event $A(z)$, each of the points $(V_1(z),H_1(z)),\ldots, (V_{d+1}(z),H_{d+1}(z))$ form a cell of $\cL(\eta)$ almost surely. Indeed, for any $(v,h)\in K_i(z)$ it holds that
\[
\Pi^{\downarrow}_{-}(v,h)=\{(w,s)\in\R^{d+1}\colon \|v-w\|^2\leq h-s\}\subset K(z).
\]
This follows since for any $(w,s)\in \Pi^{\downarrow}_{-}(v,h)$ we have that $s\leq h\leq \ell(\eta)$ and 
\[
\|w-z\|\leq \|w-v\|+\|v-z\|\leq \sqrt{h-s}+4\delta\leq \sqrt{\ell(\eta)-s}+4\delta,
\]
implying $(w,s)\in K(z)$. At the same time 
\[
\Pi^{\downarrow}_{-}(v,h)\cap K_j(z)=\emptyset
\]
for any $j\neq i$. In order to show this, we first note that for any $1\leq i<j\leq d+1$ we have $\|x_i(z)-x_j(z)\|\ge 3\sqrt{2}\delta$ by Jung's theorem. let $(w,s)\in K_j(z)$, then
\[
\|v-w\|\ge \|x_i(z)-x_j(z)\|-\|v-x_i(z)\|-\|w-x_j(z)\|\ge (3\sqrt{2}-2)\delta.
\]
On the other hand, we have $\sqrt{h-s}\leq \delta$ and hence,
\[
\|v-w\|\ge (3\sqrt{2}-2)\delta>\delta\ge \sqrt{h-s},
\]
implying $(w,s)\not\in \Pi^{\downarrow}_{-}(v,h)$. Together, this implies that for any $(V_i(z),H_i(z))$, it holds that $V_i(z)\in C((V_i(z),H_i(z)),\eta)\neq \emptyset$. Indeed, for any $(v,h)\neq (V_i(z),H_i(z))\in\eta$ we get
\[
{\rm pow}(V_i(z),(V_i(z),H_i(z)))=H_i(z)< \|V_i(z)-v\|^2+h={\rm pow}(V_i(z),(v,h)),
\]
since otherwise $(v,h)\in \Pi^{\downarrow}_{-}(V_i(z),H_i(z))$, which is a contradiction to $A(z)$.

Moreover, the point $(V_0(z),H_0(z))$ forms a cell with probability $p(\eta)=p\in (0,1)$. In order to show this, we note that for any $(v,h)\in K_0(z)$ it holds that $\Pi_{-}^{\downarrow}(v,h)\subset K(z)$. Then $(V_0(z),H_0(z))$ forms a cell if $(V_i(z),H_i(z))\not\in\Pi_{-}^{\downarrow}(V_0(z),H_0(z))$ for all $1\leq i\leq d+1$. This happens for instance for $V_0(z)=z$, $H_0(z)=\ell(\eta)-9\delta^2$, $V_i(z)=4(x_i(z)-z)/3+z$, $H_i(z)=\ell(\eta)-18\delta^2$ since in this case
\[
\|V_0(z)-V_i(z)\|=4\delta>3\delta=\sqrt{H_0(z)-H_i(z)},
\]
for any $1\leq i\leq d+1$. By continuity there is a subset $A\subset K_0(z)\times K_1(z)\times\ldots\times K_{d+1}(z)$ of positive Lebesgue measure, such that for any $(v_0,h_0,v_1,h_1,\ldots,v_{d+1},h_{d+1})\in A$, the point $(v_0,h_0)$ forms a cell. On the other hand, the point $(V_0(z),H_0(z))$ does not form a cell if for any $(w,s)$ satisfying 
\[
\|w-V_0(z)\|\leq \sqrt{s-H_0(z)},
\]
there is $1\leq i\leq d+1$ with 
\[
\|w-V_i(z)\|\leq \sqrt{s-H_i(z)}.
\]
This holds for instance when $V_0(z)=z$, $H_0(z)=\ell(\eta)$, $V_i(z)=2(x_i(z)-z)/3+z$, $H_i(z)=\ell(\eta)-19\delta^2$. In this case, let $V_j(z)$ be such that $\|w-V_j(z)\|=\min_{1\leq i\leq d+1}\|w-V_i(z)\|$. Then 
\begin{align*}
\|V_j(z)-w\|&\leq \max(\|V_0(z)-w\|,\|V_0(z)-V_j(z)\|)\\
&\leq \max(\sqrt{s-H_0(z)}, \sqrt{H_0(z)-H_j(z)})\leq \sqrt{s-H_j(z)},
\end{align*}
where the second inequality holds since $\|V_0(z) -V_j(z)\|=2\delta<\sqrt{19}\delta=\sqrt{H_0(z)-H_j(z)}$ and the last inequality follows since $H_0(z)\leq s$ and $H_j(z)=\ell(\eta)-19\delta^2<\ell(\eta)=H_0(\eta)$.
Again by continuity there is a subset $A'\subset K_0(z)\times K_1(z)\times\ldots\times K_{d+1}(z)$ of positive Lebesgue measure, such that for any $(v_0,h_0,v_1,h_1,\ldots,v_{d+1},h_{d+1})\in A'$ the point $(v_0,h_0)$ does not form a cell. Finally condition on $A(z)$ the point $(V_i(z),H_i(z))$ is distributed in $K_i(z)$ according to $\Q/\Q(K_i(z))$ for any $0\leq i\leq d+1$ and $(V_0(z),H_0(z)),\ldots, (V_{d+1}(z),H_{d+1}(z))$ are independent (follows e.g., by \cite[Theorem 3.2.2]{SW}). Hence, since $\Q$ has a density with respect to Lebesgue measure
\[
p=\P((V_0(z),H_0(z))\text{ forms a cell})\ge {\Q^{\otimes (d+2)}(A)\over \Q(K_0(z))\cdot\ldots\cdot\Q(K_{d+1}(z))}>0,
\]
and 
\[
1-p=\P((V_0(z),H_0(z))\text{ does not form a cell})\ge {\Q^{\otimes (d+2)}(A')\over \Q(K_0(z))\cdot\ldots\cdot\Q(K_{d+1}(z))}>0,
\]
which implies $p\in (0,1)$.

Finally, for two different sets $K(z_1)$ and $K(z_2)$ with $\ind(A(z_1))=\ind(A(z_2))=1$, the events $\{(V_0(z_1),H_0(z_1))\text{ forms a cell}\}$ and $\{(V_0(z_2),H_0(z_2))\text{ forms a cell}\}$ are independent. This follows from the fact that $K(z_1)\cap K_i(z_2)=\emptyset$, $0\leq i\leq d+1$ for any $z_1\neq z_2$ if $\delta=1/16$.

Hence,
\begin{align*}
\E(\Var(F_n(\eta)|\mathcal{G}))&=\E\sum_{z\colon \ind(A(z))=1}\Var_{(V_0(z),H_0(z),\ldots,V_{d+1}(z), H_{d+1}(z))}(F_n(\eta)|\mathcal{G})\\
&\ge \E\sum_{z\colon \ind(A(z))=1}p(1-p)\ge C n^d,
\end{align*}
since we have either $d+1$ cells with probability $1-p(\eta)$ (i.e.\ $(V_0(z), H_0(z))$ does not form a cell, but $(V_i(z),H_i(z))$, $1\leq i\leq d+1$ do form cells) or $d+2$ cells with probability $p(\eta)$ (i.e. $(V_0(z), H_0(z))$ forms an additional cell). The proof follows.
\end{proof}

\section{Proof of Theorem \ref{thm:1}}\label{sec:proof1.1}

Below, we first list some simple integral bounds that are straightforward to check and which will be used throughout this section. 

We will often use that for constants $b \in \R$ and $a, c \neq 0$ with $(b+1)/a>0$,
\begin{equation}\label{eq:hint}
	\int_{0}^{\infty} h^b\exp(-c h^a)\md h \lesssim \int_0^\infty w^{\frac{b+1}{a} -1}\exp(-c w) \md w<\infty.
\end{equation}

Another simple fact, which will be useful for us in the case of the $\widetilde \eta$ models is the following: for $a,b,c,k>0$, one has
\begin{equation}\label{eq:hintgaussian}
	\int_{\R} ((1-h) \vee 1)^{k}e^{ah} \exp(-b e^{ch}) \md h <\infty.
\end{equation}

Also, for any $v_2 \in \R^d$ and $a, c>0$, 
\begin{equation}\label{eq:vint}
	\int_{\R^d}\exp(-c \|v_1-v_2\|^a) \md v_1 <\infty.
\end{equation}

Moreover, note that for $v_1 \in W_n$ and $v_2 \in W_{2\sqrt{d} n}^c$, it holds that
\begin{equation}\label{eq:te}
    \|v_2 - v_1\| \ge \|v_2\| - \|v_1\| \ge \|v_2\|/2
\end{equation}
Thus, we have
\begin{align}\label{eq:vintwn}
	&\int_{W_n} \exp(-c \|v_1-v_2\|^a) \md v_1 \nonumber\\
    &\lesssim \left[\ind(v_2\in W_{2\sqrt{d}n})+ \ind(v_2\in W_{2\sqrt{d}n}^c) \exp(-(c/2^{a+1})\|v_2\|^a)\right] \int_{W_n} \exp(-(c/2) \|v_1-v_2\|^a) \md v_1 \nonumber\\
    &\lesssim \ind(v_2\in W_{2\sqrt{d}n})+ \ind(v_2\in W_{2\sqrt{d}n}^c) \exp(-(c/2^{a+1})\|v_2\|^a).
\end{align}

\bigskip

Now we are ready for the proof of Theorem \ref{thm:1}. In the setting of Theorem \ref{thm:KolBd}, we fix $\X=\R^d \times \R$, equipped with the Euclidean metric on $\R^{d+1}$, with an associated measure $\Q(\md v, \md h)$ absolutely continuous w.r.t.\ the Lebesgue measure, which we will take as the intensity measure of the processes $\eta_\beta, \eta'_\beta$ or $\widetilde \eta$. Noting that a point $x=(v,h)$ forms a cell in $\eta \in \{\eta_\beta, \eta'_\beta,\widetilde \eta\}$ if and only if it is an extreme point, we consider the score function $\xi_n: \X \times \bN \to \R$ given by
$$
\xi_n((v,h), \cM):= \ind (v \in W_n) \ind ((v,h) \text{ is an extreme point in $\cM$}),
$$
for $\cM \in \bN$ and $(v,h) \in \cM$. Thus our functional of interest, that is $F_n(\eta)$ in Theorem \ref{thm:1} can be expressed as a sum as in \eqref{eq:scoresum} with the scores being $\xi_n$. To apply Theorem \ref{thm:KolBd}, we then only need to check conditions (A0) - (A2). One can readily check that this score function satisfies the monotonicity condition \eqref{eq:ximon} in assumption (A0). Next, we define a region of stabilization for the scores $\xi_n$. For $\cM \in \bN$ and $(v,h) \in \cM$, letting $T=T((v,h),\cM)$ be the cover time defined at \eqref{eq:covertime1}, we can define the region of stabilization for the scores $\xi_n$ as
\begin{equation}
	R_n((v,h), \cM): = \begin{cases}
		D((v,h), T((v,h), \cM)),\,& \text{if} \quad T>h, v \in W_n,\\
		\emptyset,\, & \text{otherwise},
	\end{cases}
\end{equation}
where $D((v,h), T((v,h), \cM))$, for $(v,h)$ with $T>h$ is given by (see also Figure \ref{Fig:RS})
$$
D((v,h), T((v,h), \cM)):=\{(v',h') \in \R^{d+1}: h'\le T, \sqrt{T-h} + \sqrt{T-h'} \ge \|v - v'\|\}.
$$
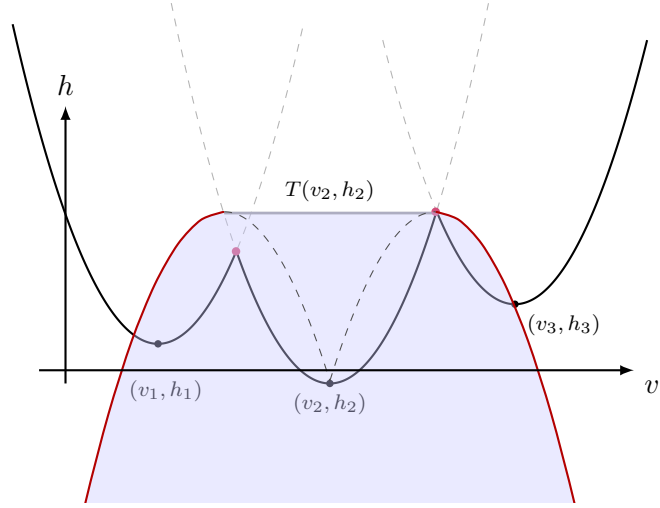
\begin{figure}
\centering
		\begin{tikzpicture}[x=50pt,y=50pt,yscale=1,xscale=1,>=latex]
					\path (-4.5,3);				
	\coordinate (L) at (-1.3,0.2);
	\coordinate (M) at (0,-0.1);
	\coordinate (R) at (1.4,0.5);
	
	\draw[thick,domain=-2.4:-0.7,samples=60,smooth,variable=\y]
	plot({\y},{0.2 + 2*(\y + 1.3)^2});
	\draw[gray!60,dashed,domain=-0.7:-0.2,samples=60,smooth,variable=\y]
	plot({\y},{0.2 + 2*(\y + 1.3)^2});
	\draw[thick,domain=-0.71:0.81,samples=60,smooth,variable=\y]
	plot({\y},{-0.1 + 2*(\y - 0)^2});
	\draw[gray!60,dashed,domain=-1.2:-0.71,samples=40,smooth,variable=\y]
	plot({\y},{-0.1 + 2*(\y - 0)^2});
	\draw[gray!60,dashed,domain=0.81:1.2,samples=40,smooth,variable=\y]
	plot({\y},{-0.1 + 2*(\y - 0)^2});
	\draw[thick,domain=0.81:2.4,samples=60,smooth,variable=\y]
	plot({\y},{0.5 + 2*(\y - 1.4)^2});
	\draw[gray!60,dashed,domain=0.4:0.81,samples=40,smooth,variable=\y]
	plot({\y},{0.5 + 2*(\y - 1.4)^2});
	
	\draw[gray!80,line width=1pt] (-0.8,1.19) -- (0.8,1.19);
	
	\fill[black] (L) circle (1.4pt);
	\fill[black] (M) circle (1.4pt);
	\fill[black] (R) circle (1.4pt);
	\node[below] at (-1.24,0) {\scriptsize $(v_1,h_1)$};
	\node[below] at (M) {\scriptsize $(v_2,h_2)$};
	\node[below] at (1.77,0.5) {\scriptsize $(v_3,h_3)$};
	
\coordinate (A) at (-0.71,.9);
\coordinate (B) at (0.8,1.2);
\fill[purple!70] (A) circle (1.6pt);
\fill[purple!70] (B) circle (1.6pt);

	\def\T{1.2} 
	\node[above] at (0,1.2) {\scriptsize $T(v_2,h_2)$};
	\begin{scope}
		\clip (-3,-1) rectangle (5,1.2); 
		
		\fill[blue!20,opacity=0.4]
		plot[domain=-5:-0.8,smooth,variable=\y] ({\y},{\T - 2*(\y + 0.8)^2})
		-- (-0.8,\T) -- (0.8,\T)
		-- plot[domain=0.8:5,smooth,variable=\y] ({\y},{\T - 2*(\y - 0.8)^2})
		-- (5,-1.01) -- (-5,-1.01) -- cycle;
		
		\draw[red!70!black,thick]
		plot[domain=-5:-0.8,smooth,variable=\y] ({\y},{\T - 2*(\y + 0.8)^2});
		\draw[red!70!black,thick]
		plot[domain=0.8:5,smooth,variable=\y] ({\y},{\T - 2*(\y - 0.8)^2});
		
		\draw[gray!60!black,dashed]
		plot[domain=-0.8:0.0,smooth,variable=\y] ({\y},{\T - 2*(\y + 0.8)^2});
		\draw[gray!60!black,dashed]
		plot[domain=0.0:0.8,smooth,variable=\y] ({\y},{\T - 2*(\y - 0.8)^2});
		\end{scope}
			
			\draw[->,thick] (-2.2,0) -- (2.3,0) node[below right]{$v$};
			\draw[->,thick] (-2.0,-0.1) -- (-2.0,2) node[above]{$h$};
		\end{tikzpicture}
        \caption{Illustration of the region $D((v_2,h_2), T(v_2,h_2),\cM)$ for some point configuration $\cM$ containing these three points such that the height $T(v_2, h_2)$ is the coverage time of the cell of $(v_2,h_2)$ defined at \eqref{eq:covertime1}}\label{Fig:RS}
        \end{figure}
Again, it is straightforward to see that $R_n$ satisfies conditions (A1.1) - (A1.4) in Theorem \ref{thm:KolBd}. Also, condition (A2) therein is satisfied with any $p>0$ and $M_p =1$. Thus, Theorem \ref{thm:KolBd} applies, and we need to upper bound the quantities on the right-hand side of the bounds therein to prove our claim.

To that end, write $x,y,z \in \X$ as $x = (v_1,h_1), y = (v_2,h_2), z=(v_3,h_3)$. We start by bounding $\P(x \in R_n(y, \eta + \delta_{y}))$. Note in this case that we must have $ \emptyset \neq R_n(y, \eta + \delta_{y}) \ni x$, so that
$$
\|v_1-v_2\|\leq 2(\sqrt{T(y, \eta + \delta_{y})-h_1}\vee \sqrt{T(y, \eta + \delta_{y})-h_2})=2\sqrt{T(y, \eta + \delta_{y})-(h_1 \wedge h_2)},
$$ 
which implies that 
$$
T(y, \eta + \delta_{y}) \ge \|v_1-v_2\|^2/4+(h_1 \wedge h_2).
$$
Together with the observation that we also have $T(y, \eta + \delta_{y}) > h_1 \vee h_2$, this implies that
$
T(y, \eta + \delta_{y})\ge \|v_1-v_2\|^2/8+h_1/2+h_2/2.
$
Thus, we obtain for $x,y \in \X$ that
\begin{align}\label{eq:tailbd}
	&\P(x \in R_n(y, \eta + \delta_{y}))\nonumber\\
    &\le \ind(v_2 \in W_n)\P\left(T(y, \eta +\delta_{y})>h_1 \vee h_2, 
    T(y, \eta + \delta_{y}) >\frac{\|v_1 - v_2\|^2}{4}+ {h_1\wedge h_2}\right)\nonumber\\
    &\le \ind(v_2 \in W_n)\P\left(T(y, \eta +\delta_{y})>h_1 \vee h_2, 
    T(y, \eta + \delta_{y}) >\frac{\|v_1 - v_2\|^2}{8}+ \frac{h_1}{2}+\frac{h_2}{2}\right).
\end{align}
Arguing similarly, when $x,y \in R_n(z, \eta + \delta_{z})$, one must have that $T(z, \eta + \delta_{z}) > h_1 \vee h_2 \vee h_3$ as well as 
$T(z, \eta + \delta_{z}) > \max\{\|v_1-v_3\|^2/4+(h_1\wedge h_3), \|v_2-v_3\|^2/4+(h_2\wedge h_3)\}$, which also implies $T(z,\eta+\delta_z)>\|v_1-v_3\|^2/16+\|v_2 - v_3\|^2/16+ h_1/4 + h_2/4+h_3/2$. Hence for $x,y,z \in \X$ we have
\begin{align}
	\P(&x,y \in R_n(z, \eta + \delta_{z})) \le  \ind(v_3 \in W_n) \P\Big(T(z, \eta + \delta_{z}) > h_1 \vee h_2 \vee h_3,\notag\\
      &\qquad \qquad T(z, \eta + \delta_{z}) >\frac{\|v_1 - v_3\|^2}{4} + (h_1\wedge h_3),T(z, \eta + \delta_{z}) >\frac{\|v_2 - v_3\|^2}{4} + (h_2\wedge h_3)\Big)\notag\\
	&\le \ind(v_3 \in W_n) \P\Big( T(z, \eta + \delta_{z}) > h_1 \vee h_2 \vee h_3, \notag\\
    &\qquad \qquad T(z, \eta + \delta_{z}) >\frac{\|v_1 - v_3\|^2}{16} + \frac{\|v_2 - v_3\|^2}{16}+ \frac{h_1}{4} + \frac{h_2}{4}+\frac{h_3}{2}\Big).\label{eq:qjoint}
\end{align}
With these bounds, we now proceed to bound the summands appearing in the bounds in Theorem \ref{thm:KolBd} separately for the three cases of $\eta_\beta, \eta'_\beta$ and $\widetilde \eta$. In the following proofs, constants such as $c,c',C \in (0,\infty)$ may change from line to line without being mentioned, but they are always finite and do not depend on $n$. Throughout, we denote the points in $\X$ by $x = (v_1,h_1), y = (v_2,h_2), z =(v_3,h_3)$ etc.

\medskip

\noindent \underline {\em Proof of Theorem \ref{thm:1} for $\eta_\beta$:} For this model, we have $\Q (\md v, \md h)= \Q_\beta (\md v, \md h)= \gamma c_{d,\beta}h^\beta \md v \md h$ for $\beta>-1$ and $\gamma>0$. By \eqref{eq:tailbd} (dropping the condition $T(y,\eta_{\beta}+\delta_y)>h_1\vee h_2$) and Proposition \ref{lm:tailbounds}, we have that  there exists constants $c,c' \in (0, \infty)$ such that for $x,y \in \X$,
\begin{align*}
&\P(x \in R_n(y, \eta_\beta + \delta_{y})) \\
&\lesssim \ind(v_2 \in W_n)\exp\left(-c (h_2 +h_1 + \|v_1-v_2\|^2)^{d/2+\beta+1}\right)\\
& \le \ind(v_2 \in W_n)\exp\left(-c' (h_2^{d/2+\beta+1} +h_1^{d/2+\beta+1} + \|v_1-v_2\|^{2(d/2+\beta+1)})\right).
\end{align*}
Thus, integrating using \eqref{eq:hint} and \eqref{eq:vint}, it follows that the function $g$ in \eqref{eq:g} can be trivially bounded by a finite positive constant. Since $M_p = 1$, we have that $G$ in \eqref{eq:g} is bounded by some finite positive constant, so that $G \lesssim 1$.  

For $f_\alpha^{(1)}$ with $\alpha>0$ defined at \eqref{eq:fal}, using \eqref{eq:hint} in the second step and \eqref{eq:vintwn} in the third, we thus obtain for $y \in \X$ that
\begin{align}\label{eq:f1bd}
	f_\alpha^{(1)}(y) &\lesssim \int_\X \ind(v_1 \in W_{n})\exp\left(-c'\alpha (h_2^{d/2+\beta+1} +h_1^{d/2+\beta+1} + \|v_1-v_2\|^{2(d/2+\beta+1)})\right) \mathbb{Q}(\md x)\nonumber\\
	& \lesssim e^{-c'\alpha h_2^{d/2+\beta+1}} \int_{W_{n}} \exp\left( -c'\alpha\|v_1-v_2\|^{2(d/2+\beta+1)}\right) \md v_1\nonumber\\
	& \lesssim e^{-c'\alpha h_2^{d/2+\beta+1}} \left[ \ind(v_2 \in W_{2\sqrt{d}n}) + e^{-c''\alpha \|v_2\|^{2(d/2+\beta+1)}} \right].
\end{align}
This yields, upon using \eqref{eq:hint} and \eqref{eq:vint}, that for any $\alpha>0$,
\begin{equation}\label{eq:b1}
    \mathbb{Q}f_{\alpha}^{(1)} , \mathbb{Q}(f_{\alpha}^{(1)} )^2 \lesssim n^d.
\end{equation}

Similarly, for any $\alpha>0$, again using \eqref{eq:hint} and \eqref{eq:vint}, for $y \in \X$ we have
\begin{align*}
	f_\alpha^{(2)}(y) &\lesssim \int_\X \ind(v_2 \in W_{n})\exp\left(-c'\alpha (h_2^{d/2+\beta+1} +h_1^{d/2+\beta+1} + \|v_1-v_2\|^{2(d/2+\beta+1)})\right) \mathbb{Q}(\md x)\\
	& \lesssim e^{-c'\alpha h_2^{d/2+\beta+1}} \ind(v_2 \in W_{n}) \int_{\R^d} \exp\left( -c'\alpha\|v_1-v_2\|^{2(d/2+\beta+1)}\right) \md v_1\\
	& \lesssim e^{-c'\alpha h_2^{d/2+\beta+1}} \ind(v_2 \in W_{n}).
\end{align*}
By \eqref{eq:hint}, we thus obtain
\begin{equation}\label{eq:b2}
    \mathbb{Q}f_{\alpha}^{(2)} , \mathbb{Q}(f_{\alpha}^{(2)} )^2 \lesssim n^d.
\end{equation}

Next applying \eqref{eq:qjoint} (and dropping the condition $T(z,\eta_{\beta}+\delta_z)>h_1\vee h_2\vee h_3$) along with the fact that 
\begin{align*}
    &\|v_1-v_3\|^2+\|v_2-v_3\|^2 \ge \frac{1}{2}(\|v_1-v_3\|^2+\|v_2-v_3\|^2) + \frac{1}{4}\|v_2 - v_3\|^2 \\
    &\ge \frac{1}{4}(\|v_1-v_3\|+\|v_2-v_3\|)^2 + \frac{1}{4}\|v_2 - v_3\|^2\ge \frac{1}{4}(\|v_1 - v_2\|^2 + \|v_2 - v_3\|^2)
\end{align*}
for the first inequality and \eqref{eq:hint} and \eqref{eq:vint} for the second, for $x,y \in \X$ we can bound
\begin{align*}
	&q(x,y) \lesssim \int_\X \ind(v_3 \in W_{n}) \\
	&\,\,\, \times \exp\big(-c' (h_1^{d/2+\beta+1} +h_2^{d/2+\beta+1} +h_3^{d/2+\beta+1}+  \|v_1-v_2\|^{2(d/2+\beta+1)} +  \|v_2-v_3\|^{2(d/2+\beta+1)})\big) \mathbb{Q}(\md z)\\
	& \qquad \quad \lesssim e^{-c' (h_1^{d/2+\beta+1} +h_2^{d/2+\beta+1} +  \|v_1-v_2\|^{2(d/2+\beta+1)} )}.
\end{align*}
Thus, again by \eqref{eq:hint}, for $y \in \X$ we have
\begin{align*}
	f_\alpha^{(3)}(y) & \lesssim \int_{\X} q(x,y)^\alpha \Q(\md x)
	 \lesssim e^{-c' \alpha h_2^{d/2+\beta+1}} \int_{W_n} \exp\left(-c'  \alpha \|v_1-v_2\|^{2(d/2+\beta+1)}\right) \md v_1,
\end{align*}
so that we can argue as in \eqref{eq:f1bd} for $f_\alpha^{(1)}$ to obtain
\begin{equation}\label{eq:b3}
    \mathbb{Q}f_{\alpha}^{(3)} , \mathbb{Q}(f_{\alpha}^{(3)} )^2 \lesssim n^d.
\end{equation}

We proceed with estimating the other terms in the bound in Theorem \ref{thm:KolBd}. Let us next bound $\mathbb{Q} \kappa^\alpha$. From Proposition \ref{lm:tailbounds}, for $x=(v_1,h_1) \in \X$, we have
\begin{align*}
	\kappa(x)&=\P(\xi_n((v_1,h_1),\eta_\beta \cup \{(v_1,h_1)\}) \neq 0) \\
	&\le \ind(v_1 \in W_n)\P(T((v_1,h_1),\eta_\beta \cup \{(v_1,h_1)\}) >h_1)\lesssim \ind(v_1 \in W_n) \exp \left(-c' h_1^{d/2+\beta+1}\right).
\end{align*}
Thus, noting that $G \lesssim 1$ and integrating using \eqref{eq:hint} yield $\mathbb{Q} \kappa^\alpha G \lesssim n^d$. Also, noting that $g \lesssim f_\zeta^{(1)}$, we can use \eqref{eq:f1bd} to conclude that $\mathbb{Q} g^\alpha G \lesssim n^d$. In particular, this also yields by Theorem \ref{thm:KolBd} that $\Var(F_n(\eta_\beta)) \lesssim n^d$.

Finally, by Proposition \ref{prop:VarianceBound}, we have that there exists a constant $C\in (0,\infty)$ depending only on $d, \beta, \gamma$ such that for all $n \in \N$,
\begin{equation}\label{eq:Varbd}
	\operatorname{Var} F_n(\eta_\beta) \ge C n^{d}.
\end{equation}
This yields the first assertion for the variance. Moreover, plugging into Theorem \ref{thm:KolBd} this variance lower bound, the upper bounds in \eqref{eq:b1}, \eqref{eq:b2} and \eqref{eq:b3}, along with the bounds $\mathbb{Q} \kappa^\alpha G, \mathbb{Q} g^\alpha G$ above yield the second assertion. \qed

\medskip

\noindent \underline {\em Proof of Theorem \ref{thm:1} for $\tilde\eta$:} The proof for $\tilde \eta$ can be argued very similarly to $\eta_\beta$. As before, we fix $\X=\R^d \times \R$, equipped with the Euclidean metric on $\R^{d+1}$, now with the associated intensity measure $\Q(\md v, \md h):=e^h \md h \md v$. For this proof, we will use instead of Proposition \ref{lm:tailbounds} the unified but looser bound: there exist a constant $c \in (0,\infty)$ such that
\begin{equation}\label{eq:loosebd}
    \P(T(x,\widetilde \eta\cup\{x\})> H)
    	\lesssim ((1-h) \vee 1)^{d/2} \exp\big(-c e^{(H\vee h)/2}\big),\quad \forall\, x\in \X, H, h \in \R.
\end{equation}
We thus have by \eqref{eq:tailbd}, \eqref{eq:loosebd} and the Cauchy-Schwarz inequality that for $x,y \in \X$,
\begin{align*}
	&\P(x \in R_n(y, \widetilde\eta + \delta_{y})) \\
	&\le \ind(v_2 \in W_n)  \P\left(T(y, \widetilde\eta + \delta_{y}))>h_1 \vee h_2, T(y, \widetilde \eta + \delta_{y}) > \frac{\|v_1 - v_2\|^2}{8}+ \frac{h_1}{2}+\frac{h_2}{2} \right)\\
	&\lesssim \ind(v_2 \in W_n) ((1-h_2) \vee 1)^{d/2} \exp\left(-c' \left[e^{h_1/2}+e^{h_2/2}+ e^{\frac{\|v_1 - v_2\|^2}{16}+ \frac{h_1}{4}+\frac{h_2}{4}} \right]\right).
\end{align*}
Using $e^{x} \ge 1+x \ge x^a$ for $x \ge 0, a \in (0,1]$,  which implies $e^{-x}\leq x^{-a}$ for $x>0$, and setting $x=c' e^{\frac{\|v_1 - v_2\|^2}{16}+ \frac{h_1}{4}+\frac{h_2}{4}}>0$, $a=2/5$, we obtain the bound 
\[
\exp\left(-c' e^{\frac{\|v_1 - v_2\|^2}{16}+ \frac{h_1}{4}+\frac{h_2}{4}}\right) \lesssim e^{-\frac{\|v_1 - v_2\|^2}{40}} e^{-\frac{h_1+h_2}{10}}
\]
Then for $x \in \X$,
\begin{align}\label{eq:gbd}
	g(x) & \lesssim \int ((1-h_2) \vee 1)^{\zeta d/2} \exp\left(-c'\zeta \left[e^{h_1/2}+e^{h_2/2}+ e^{\frac{\|v_1 - v_2\|^2}{16}+ \frac{h_1}{4}+\frac{h_2}{4}} \right]\right) \Q(\md y)\nonumber\\
	& \lesssim \int_{\R} ((1-h_2) \vee 1)^{\zeta d/2}\exp\left(-c'\zeta [e^{h_1/2}+e^{h_2/2}]\right)e^{-(h_1+h_2)/10} e^{h_2} \md h_ 2 \lesssim e^{-h_1/10},
\end{align}
where the second inequality is due to \eqref{eq:vint} while the final one is due to \eqref{eq:hintgaussian}. 

Thus, noting that $G \lesssim 1+g^5$, with $f_\alpha^{(1)}$ as in \eqref{eq:fal}, using the simple inequality $e^{-x} \le x^{-1}$ for $x=c'\alpha  e^{\frac{\|v_1 - v_2\|^2}{16}+ \frac{h_1}{4}+\frac{h_2}{4}} > 0$ in the second step, for $y \in \X$ we obtain
\begin{align}\label{eq:f1bdtilde}
	f_\alpha^{(1)}(y)
    &\lesssim \int_\X \ind(v_1 \in W_{n})  (1+e^{-h_1/2}) ((1-h_1) \vee 1)^{\alpha d/2} \nonumber\\
    & \qquad \qquad \qquad \times \exp\left(-c'\alpha \left[e^{h_1/2}+e^{h_2/2}+ e^{\frac{\|v_1 - v_2\|^2}{16}+ \frac{h_1}{4}+\frac{h_2}{4}} \right] \right) \mathbb{Q}(\md x)\nonumber\\
	& \lesssim \exp(-c'\alpha e^{ h_2/2})\int (1+e^{-h_1/2}) ((1-h_1) \vee 1)^{\alpha d/2}\exp(-c'\alpha e^{h_1/2}) e^{h_1} \md h_1  \nonumber\\
    & \qquad \qquad \qquad \times\int_{W_{n}}  e^{ -\frac{\|v_1 - v_2\|^2}{16} - \frac{h_1}{4} - \frac{h_2}{4}} \md v_1\nonumber\\
	& \lesssim e^{-h_2/4} \exp(-c' \alpha e^{ h_2/2})\int  (1+e^{-h_1/2}) ((1-h_1) \vee 1)^{\alpha d/2}  \exp(-c' \alpha e^{h_1/2}) e^{3h_1/4} \md h_1  \nonumber\\
    & \qquad \qquad \qquad \times\int_{W_{n}} e^{-\frac{\|v_1 - v_2\|^2}{16}} \md v_1\nonumber\\
	& \lesssim e^{-h_2/4} \exp(-c' \alpha e^{ h_2/2}) \left[ \ind(v_2 \in W_{2\sqrt{d}n}) + e^{-c''\|v_2\|^2} \right],
\end{align}
where the final step is due to \eqref{eq:hintgaussian} and \eqref{eq:vintwn}. This yields upon integrating w.r.t.\ $\mathbb{Q}$ using \eqref{eq:hintgaussian} and \eqref{eq:vint} that for any $\alpha>0$,
\begin{equation}\label{eq:G1}
    \mathbb{Q}f_{\alpha}^{(1)} , \mathbb{Q}(f_{\alpha}^{(1)} )^2 \lesssim n^d.
\end{equation}

Arguing similarly, integrating w.r.t.\ $h_1$ and using $e^{-x} \le x^{-1}$ for $x \ge 0$ for the second inequality, for any $\alpha>0$ and $y \in \X$, we obtain
\begin{align*}
	f_\alpha^{(2)}(y) &\lesssim \int_\X \ind(v_2 \in W_{n})  (1+e^{-h_1/2})  ((1-h_2) \vee 1)^{\alpha d/2}\\
    & \qquad \qquad \qquad \times \exp\left(-c' \alpha \left[e^{h_1/2}+e^{h_2/2}+ e^{\frac{\|v_1 - v_2\|^2}{16}+ \frac{h_1}{4}+\frac{h_2}{4}} \right] \right)  \mathbb{Q}(\md x)\\
	& \lesssim e^{-h_2/4}((1-h_2) \vee 1)^{\alpha d/2} \exp(-c'\alpha e^{ h_2/2}) \ind(v_2 \in W_{n}) \int_{\R^d} e^{-\frac{\|v_1 - v_2\|^2}{16}} \md v_1\\
	& \lesssim e^{-h_2/4} ((1-h_2) \vee 1)^{\alpha d/2}\exp(-c' \alpha e^{ h_2/2}) \ind(v_2 \in W_{n}),
\end{align*}
where the second step holds by \eqref{eq:hintgaussian} and the final step is due to \eqref{eq:vint}. Upon integrating and using \eqref{eq:hintgaussian}, we get
\begin{equation}\label{eq:G2}
    \mathbb{Q}f_{\alpha}^{(2)} , \mathbb{Q}(f_{\alpha}^{(2)} )^2 \lesssim n^d.
\end{equation}

Next, from \eqref{eq:qjoint} with the Cauchy-Schwarz inequality and arguing as above using $e^{-x} \le x^{-1}$ for $x \ge 0$,  \eqref{eq:hintgaussian} and \eqref{eq:vint}, we observe for $x,y \in \X$ that
\begin{align*}
q(x,y) 
	 &\lesssim \exp\left(-c' [e^{h_1/2}+e^{h_2/2}]\right) \\
	 & \quad \times\int ((1-h_3) \vee 1)^{d/2}\exp\left(-c' e^{h_3/2}\right) e^{h_3}\md h_3 \int_{W_n} \exp\big(-c' e^{\frac{\|v_1 - v_3\|^2}{32} + \frac{\|v_2 - v_3\|^2}{32}+ \frac{h_1}{8} + \frac{h_2}{8}+\frac{h_3}{4}}\big) \md v_3 \\
	 &\lesssim \exp\left(-c' [e^{h_1/2}+e^{h_2/2}]\right) e^{-(h_1 +h_2)/8} \int_{\R^d} e^{-\|v_1-v_2\|^2/32 - \|v_1 - v_3\|^2/32} \md v_3\\
	 & \lesssim \exp\left(-c' [e^{h_1/2}+e^{h_2/2}]\right) e^{-(h_1 +h_2)/8}  e^{-\|v_1-v_2\|^2/32}.
\end{align*}
Thus an integration yields that for $y \in \X$,
\begin{align*}
	f_\alpha^{(3)}(y) &\lesssim \exp(-c' \alpha e^{ h_2/2})e^{-\alpha h_2/8} \int (1+e^{-h_1/2}) \exp(-c' \alpha e^{h_1/2}) e^{(1-\alpha/8)h_1} \md h_1  \int_{W_{n}} e^{-\alpha\|v_1 - v_2\|^2/32} \md v_1\\
     &\lesssim \exp(-c' \alpha e^{ h_2/2})e^{-\alpha h_2/8}\left[ \ind(v_2 \in W_{2\sqrt{d}n}) + e^{-c''\|v_2\|^2} \right],
\end{align*}
where we have used \eqref{eq:hintgaussian} and \eqref{eq:vintwn} for any $\alpha \in (0,1)$. Hence by \eqref{eq:hintgaussian} and \eqref{eq:vint} for any $\alpha\in (0,1)$ we get
\begin{equation}\label{eq:G3}
    \mathbb{Q}f_{\alpha}^{(3)} , \mathbb{Q}(f_{\alpha}^{(3)} )^2 \lesssim n^d.
\end{equation}

Next, from \eqref{eq:loosebd}, for $x=(v_1,h_1) \in \X$, we have
\begin{align*}
	\kappa(x)&=\P(\xi_n((v_1,h_1),\tilde \eta \cup \{(v_1,h_1)\}) \neq 0) \\
	&\lesssim \ind(v_1 \in W_n)\, \P(T((v_1,h_1),\tilde \eta \cup \{(v_1,h_1)\}) >h_1) \lesssim \ind(v_1 \in W_n) ((1-h_1) \vee 1)^{d/2}\exp\big(-c e^{h_1/2}\big).
\end{align*}
Thus, recalling that $G(x) \lesssim 1+e^{-h_1/2}$, integration yields $\mathbb{Q} \kappa^\alpha G \lesssim n^d$. Arguing exactly similarly as for $\eta_\beta$, noting that $g \lesssim f_\zeta^{(1)}$, we can use \eqref{eq:f1bdtilde} to conclude that $\mathbb{Q} g^\alpha G \lesssim n^d$ for $\alpha \in (0,1)$. Again, this yields by Theorem \ref{thm:KolBd} that $\Var(F_n(\tilde \eta)) \lesssim n^d$.

By Proposition \ref{prop:VarianceBound}, we have that there exists a constant $C\in (0,\infty)$ depending only on $d$ such that for all $n \in \N$,
\begin{equation}\label{eq:Varbd}
	\operatorname{Var} F_n(\tilde \eta) \ge C n^{d}
\end{equation}
which yields the variance bounds as in the first assertion. Plugging these bounds above as well as in \eqref{eq:G1}, \eqref{eq:G2} and \eqref{eq:G3} into Theorem \ref{thm:KolBd} now yields the second assertion, completing the proof. \qed

\medskip

\noindent \underline {\em Proof of Theorem \ref{thm:1} for $\eta_\beta'$:} We fix $\X=\R^d \times (-\infty,0)$, equipped with the Euclidean metric on $\R^{d+1}$ and take $\Q(\md v, \md h):=\gamma c'_{d,\beta}(-h)^{-\beta} \md h \md v$
for $\beta>5d+1$ and $\gamma>0$. Here, for condition (A2) in Theorem \ref{thm:KolBd}, we take $p>0$, to be fixed later, and $M_p =1$. Thus, noting $\P(T(y,\eta'_{\beta}+\delta_y)>H)=0$ for $H\ge 0$, we have by \eqref{eq:tailbd} and Proposition \ref{lm:tailbounds} that for $x,y \in \X$,
\begin{align}\label{eq:tailbeta'}
	&\P(x \in R_n(y, \eta_\beta' + \delta_{y})) \nonumber\\
	&\le \ind(v_2 \in W_n)\P\left(T(y, \eta_\beta' + \delta_{y})> h_1 \vee h_2, T(y, \eta_\beta' + \delta_{y}) > h_1 \wedge h_2+\frac{\|v_1 - v_2\|^2}{4} \right)\nonumber\\
	&\le \ind(v_2 \in W_n) \ind \left(h_1 \wedge h_2+\frac{\|v_1 - v_2\|^2}{4}<0 \right) \P(T(y, \eta_\beta + \delta_{y})> h_1 \vee h_2)\nonumber\\
    & \lesssim  \ind(v_2 \in W_n)  \ind \left(h_1 \wedge h_2+\frac{\|v_1 - v_2\|^2}{4}<0 \right)\exp\left(-c' \left(-(h_1 \vee h_2)\right)^{-(\beta-d/2-1)}\right).
\end{align}
By a slight abuse of notation, replacing $h \in (-\infty,0)$ by $-h \in (0,\infty)$, we write the integrals below over $\X':=\R^d \times (0,\infty)$ with the correspondingly updated associated measure $\Q'(\md v, \md h):=\gamma c'_{d,\beta} h^{-\beta} \md h \md v$.

With $g$ as in \eqref{eq:g}, integrating w.r.t.\ $v_2$ in the first step, we have for $x \in \X'$ that
\begin{align*}
	g(x) &\lesssim \int_{\X'} \ind(h_1 \vee h_2 > \|v_1 - v_2\|^2/4) \exp\left(-c' \zeta \left(h_1 \wedge h_2\right)^{-(\beta-d/2-1)}\right) \mathbb{Q}'(\md y)\\
	& \lesssim \int_0^\infty (h_1 \vee h_2)^{d/2}  \exp\left(-c' \zeta (h_1 \wedge h_2)^{-(\beta-d/2-1)}\right) h_2^{-\beta} \md h_2\\
	& =  \int_{h_1}^\infty h_2^{d/2-\beta} \exp\left(-c' \zeta h_1^{-(\beta-d/2-1)}\right)  \md h_2 + h_1^{d/2} \int_0^{h_1} \exp\left(-c' \zeta h_2^{-(\beta-d/2-1)}\right) h_2^{-\beta} \md h_2\nonumber\\
	& \lesssim h_1^{-(\beta-d/2-1)} \exp\left(-c' \zeta h_1^{-(\beta-d/2-1)}\right)+ h_1^{d/2} \int_0^{h_1} \exp\left(-c' \zeta h_2^{-(\beta-d/2-1)}\right) h_2^{-\beta} \md h_2 \\
	&\lesssim 1+ h_1^{d/2} \int_0^{h_1} \exp\left(-c' \zeta h_2^{-(\beta-d/2-1)}\right) h_2^{-\beta} \md h_2 \lesssim 1+h_1^{d/2},
	\end{align*}
where the final step uses the following fact which is a consequence of \eqref{eq:hint}: For any $a>1, b>0$,
\begin{equation}\label{eq:gammaint}
	\int_0^\infty h_2^{-a} \exp\left(-b h_2^{-(\beta-d/2-1)}\right)  \md h_2 <\infty.
\end{equation}
Thus, we have $G(x_1) \lesssim 1+g(x_1)^5 \lesssim 1+ h_1^{5d/2}$. Also, note the following bound: for $h_2>0$ and $a,b \in \R$ with $\beta>\max\{a, b, a+b\} + 1$,
\begin{align}\label{eq:intinwn}
     &\int_0^\infty (h_1 \vee h_2)^{a}  (1+h_1^{b})\exp\left(-c' (h_1 \wedge h_2)^{-(\beta-d/2-1)}\right) h_1^{-\beta} \md h_1 \nonumber\\
	& \lesssim h_2^{a} \int_0^{h_2} (1+h_1^{b}) \exp\left(-c' h_1^{-(\beta-d/2-1)}\right) h_1^{-\beta} \md h_1 \nonumber\\
	& \qquad \qquad \qquad + \int_{h_2}^\infty (1+h_1^{b}) h_1^{a-\beta} \exp\left(-c' h_2^{-(\beta-d/2-1)}\right)  \md h_1 \nonumber\\
	& \lesssim \exp\left(-(c'/2) h_2^{-(\beta-d/2-1)}\right)  \Bigg[h_2^{a} \int_0^{h_2} (1+h_1^{b}) \exp\left(-(c'/2) h_1^{-(\beta-d/2-1)}\right) h_1^{-\beta} \md h_1 \nonumber\\
    & \qquad \qquad+ \int_{h_2}^\infty (1+h_1^{b}) h_1^{a-\beta} \exp\left(-(c'/2) h_1^{-(\beta-d/2-1)}\right)  \md h_1  \Bigg]\nonumber\\
	&\lesssim \exp\left(-(c'/2) h_2^{-(\beta-d/2-1)}\right) ( h_2^{a} + 1),
\end{align}
where the final step is due to \eqref{eq:gammaint} and the fact that $\beta>\max\{a, b, a+b\}+1$.

To bound $f_\alpha^{(1)}(y)$, we first consider the case when $v_2 \in W_{2\sqrt{d}n}$. In this case, arguing similarly as for $g$, integrating w.r.t.\ $v_2$ and using \eqref{eq:intinwn}, we have for $y \in \X'$ with $v_2 \in W_{2\sqrt{d}n}$ that for $\beta>3d+1$
\begin{align*}
	 f_\alpha^{(1)}(y) &\lesssim \int_{\X'} \ind(h_1 \vee h_2 > \|v_1 - v_2\|^2/4) (1+h_1^{5d/2})\exp\left(-c'\alpha \left(h_1 \wedge h_2\right)^{-(\beta-d/2-1)}\right) \mathbb{Q}'(\md x) \nonumber\\
     &\lesssim \int_0^\infty (h_1 \vee h_2)^{d/2}  (1+h_1^{5d/2})\exp\left(-c'\alpha (h_1 \wedge h_2)^{-(\beta-d/2-1)}\right) h_1^{-\beta} \md h_1 \nonumber\\
	&\lesssim \exp\left(-(c'\alpha/2) h_2^{-(\beta-d/2-1)}\right) ( h_2^{d/2} + 1)\nonumber,
\end{align*}
here we trivially bounded $\ind(v_1\in W_n)\leq 1$. Thus, we obtain that
\begin{equation}\label{eq:f1b'}
	\ind(v_2 \in W_{2\sqrt{d}n})f_\alpha^{(1)}(y) \lesssim \ind(v_2 \in W_{2\sqrt{d}n})( h_2^{d/2} + 1) \exp\left(-(c'\alpha/2) h_2^{-(\beta-d/2-1)}\right).
\end{equation}

On the other hand, let $v_2 \in W_{2\sqrt{d}n}^c$. Since $v_1\in W_n$, we have by \eqref{eq:te} that $\|v_1 - v_2\| \ge \|v_2\|/2$. Thus arguing similarly as above, using \eqref{eq:gammaint} in the third step, we obtain for $y \in \X'$ with $v_2 \in W_{2\sqrt{d}n}^c$ and $\beta>5d/2+1$ that
\begin{align*}
	&f_\alpha^{(1)}(y) \lesssim \int_{\X'} \ind(v_1 \in W_n)\ind(h_1 \vee h_2 > \|v_1 - v_2\|^2/4) (1+h_1^{5d/2})\exp\left(-c'\alpha \left(h_1 \wedge h_2\right)^{-(\beta-d/2-1)}\right) \mathbb{Q}'(\md x) \nonumber\\
    &\quad \lesssim n^d \ind(h_2 \ge \|v_2\|^2/16) \int_0^{h_2} \exp\left(-c'\alpha h_1^{-(\beta-d/2-1)}\right) (1+h_1^{5d/2})h_1^{-\beta} \md h_1 \\
	& \qquad\qquad \qquad  + n^d \int_{h_2}^\infty \ind(h_1 \ge \|v_2\|^2/16) (1+h_1^{5d/2}) \exp\left(-c'\alpha h_2^{-(\beta-d/2-1)}\right) h_1^{-\beta}  \md h_1 \\
	& \quad  \lesssim n^d \exp\left(-(c'\alpha/2) h_2^{-(\beta-d/2-1)}\right) \left[\ind(h_2 \ge \|v_2\|^2/16) + (h_2 \vee \|v_2\|^2)^{1-\beta}+ (h_2 \vee \|v_2\|^2)^{5d/2+1-\beta} \right]\\
	&\quad  \lesssim n^d \exp\left(-(c'\alpha/2) h_2^{-(\beta-d/2-1)}\right) \left[\ind(h_2 \ge \|v_2\|^2/16) + (h_2 \vee \|v_2\|^2)^{5d/2+1-\beta} \right]
\end{align*}
for $n$ large enough, where the last step is due to the fact that $\|v_2\| \ge 2\sqrt{d}n$.
Combining the cases for $v_2 \in W_{2\sqrt{d}n}$ and $v_2 \in W_{2\sqrt{d}n}^c$, we obtain
\begin{multline*}
	\mathbb{Q}'f_\alpha^{(1)} \lesssim n^d \int_0^\infty ( h_2^{d/2} + 1) h_2^{-\beta} \exp\left(-(c'\alpha/2) h_2^{-(\beta-d/2-1)}\right)  \md h_2 \\
	  + n^d \int_{W_{2\sqrt{d}n}^c} \int_0^\infty \left[\ind(h_2 \ge \|v_2\|^2/16) + (h_2 \vee \|v_2\|^2)^{5d/2+1-\beta} \right]\\
      \times h_2^{-\beta} \exp\left(-(c'\alpha/2) h_2^{-(\beta-d/2-1)}\right) \md h_2\md v_2.
\end{multline*}
Thus by \eqref{eq:gammaint}, noting that $h_2 \ge \|v_2\|^2/16 \ge d n^2/4$ on $W_{2{\sqrt{d}}n}^c$, we have
\begin{align}\label{eq:b'11}
	\mathbb{Q}'f_\alpha^{(1)} & \lesssim n^d + n^d \int_{dn^2/4}^\infty h_2^{d/2-\beta} \md h_2 \nonumber\\
	& \quad + n^d \int_{W_{2\sqrt{d}n}^c} \int_0^\infty (h_2 \vee \|v_2\|^2)^{5d/2+1-\beta} h_2^{-\beta} \exp\left(-(c'\alpha/2) h_2^{-(\beta-d/2-1)}\right) \md h_2\md v_2 \nonumber\\
	& \lesssim n^d  + n^d \int_{dn^2/4}^\infty h_2^{d/2-\beta} \md h_2 + n^d \int_{4dn^2}^\infty h_2^{d/2+(5d/2+1-\beta)-\beta} \md h_2 \nonumber\\
	& \quad + n^d \int_0^\infty \int_{\|v_2\| \ge n \vee \sqrt{h_2}} \|v_2\|^{2(5d/2+1-\beta)} h_2^{-\beta} \exp\left(-(c'\alpha/2) h_2^{-(\beta-d/2-1)}\right) \md h_2 \md v_2 \nonumber\\
	& \lesssim n^d + n^d \int_{\|v_2\| \ge n} \|v_2\|^{2(5d/2+1-\beta)} \md v_2 \lesssim n^d,
\end{align}
where in the penultimate step, we have divided the last summand into two integrals, corresponding to $h_2 \ge \|v_2\|^2 \ge 4dn^2$ and $h_2 \le \|v_2\|^2$, respectively, and the first, the penultimate and the final steps use the fact that $\beta>3d+1$.  

Similarly, using again that $\beta>3d+1$ for the first as well as the final steps, we obtain
\begin{align}\label{eq:b'12}
	&\mathbb{Q}'(f_\alpha^{(1)})^2 \lesssim n^d \int_0^\infty ( h_2^{d} + 1) h_2^{-\beta} \exp\left(-c'\alpha h_2^{-(\beta-d/2-1)}\right)   \md h_2 \nonumber\\
	& \, + n^{2d} \int_{W_{2\sqrt{d}n}^c} \int_0^\infty \left[ \ind(h_2 \ge \|v_2\|^2/16) + (h_2 \vee \|v_2\|^2)^{2(5d/2+1-\beta)}\right] h_2^{-\beta}  \exp\left(-c'\alpha h_2^{-(\beta-d/2-1)}\right)  \md h_2\md v_2 \nonumber\\
	& \qquad \qquad \lesssim n^d + n^{2d} \int_{dn^2/4}^\infty h_2^{d/2-\beta} \md h_2 \nonumber\\
	& \qquad \qquad \qquad + n^{2d} \int_{4dn^2}^\infty h_2^{d/2+2(5d/2+1-\beta)-\beta} \md h_2 + n^{2d}  \int_{\|v_2\| \ge n} \|v_2\|^{4(5d/2+1-\beta)} \md v_2  \nonumber\\
    & \qquad \qquad \lesssim n^d + n^{3d+1-\beta}+n^{13d+5-6\beta}+n^{13d+4-4\beta}
  \lesssim n^d.
\end{align}

Similarly, for any $\alpha>0$, using \eqref{eq:intinwn} we have for $y \in \X'$ that
\begin{align*}
	f_\alpha^{(2)}(y) &\lesssim \int_{\X'} \ind(v_2 \in W_{n})\ind(h_1 \vee h_2 > \|v_1 - v_2\|^2/4)(1+h_1^{5d/2}) \exp\left(-c'\alpha \left(h_1 \wedge h_2\right)^{-(\beta-d/2-1)}\right) \mathbb{Q}' (\md x)\\
	& \lesssim \ind(v_2 \in W_{n}) \int_{0}^\infty (h_1 \vee h_2)^{d/2} (1+h_1^{5d/2}) \exp\left(-c'\alpha \left(h_1 \wedge h_2\right)^{-(\beta-d/2-1)}\right) h_1^{-\beta} \md h_1\\
	& \lesssim \ind(v_2 \in W_{n}) (h_2^{d/2} + 1) \exp\left(-(c'\alpha/2) h_2^{-(\beta-d/2-1)}\right).
\end{align*}
Upon integrating using \eqref{eq:gammaint}, since $\beta>d+1$, we get
\begin{equation}\label{eq:b'2}
    \mathbb{Q}'f_{\alpha}^{(2)} , \mathbb{Q}'(f_{\alpha}^{(2)} )^2 \lesssim n^d.
\end{equation}

Next, we consider the integral of $f_\alpha^{(3)}$. While the choice of the parameter $p$ didn't matter so far, here, we take $p=32$ to optimize the range of $\beta$ where our result holds. Note in this case, our $\alpha$, which is either $\tau$ or $2\tau$, satisfy $1/5 \le \alpha \le 2/5$. Also, again to optimize the range of $\beta$, we will work with a slightly refined bound on $G$, namely, we will use $G(x) \lesssim 1+ g(x)^{4+\frac{4}{4+p/2}} \lesssim 1+ g(x)^{21/5} \lesssim 1+ h_1^{21d/10}$. 

Arguing similarly as in  \eqref{eq:tailbeta'}, note that $x, y \in R_n(z, \eta_\beta' + \delta_z)$ implies that 
$$
T (z, \eta_\beta' + \delta_z)>-(h_1 \wedge h_2 \wedge h_3), \quad \text{and} \quad h_1 \vee h_3>\frac{\|v_1 - v_3\|^2}{4}, h_2 \vee h_3>\frac{\|v_2 - v_3\|^2}{4}.
$$
We thus obtain for $x,y \in \X'$ that
\begin{align}\label{eq:b'q}
	q(x,y) &\lesssim \int_{\X'} \ind(v_3 \in W_{n}) \ind \left(h_1 \vee h_3>\frac{\|v_1 - v_3\|^2}{4}, h_2 \vee h_3>\frac{\|v_2 - v_3\|^2}{4} \right)\\
	& \qquad \times \exp\left(-c' \left(h_1 \wedge h_2 \wedge h_3\right)^{-(\beta-d/2-1)}\right)\mathbb{Q}'(\md z).
\end{align}
Note that when $h_3 \le h_1 \vee h_2$, the events inside the second indicator in the integrand above imply via the triangle inequality that $\|v_1-v_2\|^2 \lesssim h_1 \vee h_2$. On the other hand, when $h_3 \ge h_1 \vee h_2$, we similarly have $h_3 \gtrsim \|v_1-v_2\|^2$.

Now, let us first consider the case when $v_2 \in W_{2\sqrt{d}n}$. Then, for $x,y \in \X'$ with $v_2 \in W_{2\sqrt{d}n}$, we can use \eqref{eq:gammaint} to bound
\begin{align*}
    &\exp\left((c'/2) \left(h_1 \wedge h_2\right)^{-(\beta-d/2-1)}\right) q(x,y) \\
    &\lesssim \ind(\|v_1 - v_2\|^2 \lesssim h_1 \vee h_2) \int_0^{h_1 \vee h_2} (h_1 \vee h_3)^{d/2} \exp\left(-(c'/2) h_3^{-(\beta-d/2-1)}\right) h_3^{-\beta} \md h_3\\
    & \qquad \qquad \qquad \qquad + \int_{h_1 \vee h_2 \vee \|v_1 - v_2\|^2}^\infty h_3^{d/2-\beta} \md h_3\\
    & \lesssim \ind(\|v_1 - v_2\|^2 \lesssim h_1 \vee h_2) (h_1 \vee h_2)^{d/2} + (h_1 \vee h_2 \vee \|v_1 - v_2\|^2)^{-(\beta-d/2-1)}\\
    & \lesssim \ind(\|v_1 - v_2\|^2 \lesssim h_1 \vee h_2) ((h_1 \vee h_2)^{d/2} + (h_1 \vee h_2)^{-(\beta-d/2-1)}) \\
    & \qquad \qquad \qquad \qquad + \ind(\|v_1 - v_2\|^2 \gtrsim h_1 \vee h_2)(\|v_1 - v_2\|^2)^{-(\beta-d/2-1)}.
\end{align*}
On the other hand, let $v_2 \in W_{2\sqrt{d}n}^c$. Again by \eqref{eq:te} we have $\|v_2- v_3\| \ge \|v_2\|/2$. Note thus that when $h_3 \ge h_1 \vee h_2$, the event inside the indicator in the integrand in \eqref{eq:b'q} imply that $h_3 = h_2 \vee h_3 \gtrsim \|v_2-v_3\|^2 \gtrsim \|v_2\|^2$. This also further imply that $h_3 \gtrsim 
\gtrsim \|v_2\|^2 \vee \|v_1 - v_2\|^2 \gtrsim \|v_1\|^2$. First, let $h_1 \le h_2$. On the other hand, when $h_3 \le h_1 \vee h_2$, one has $\|v_2\|^2 \lesssim \|v_2 - v_3\|^2 \lesssim h_2 \vee h_3 = h_2$ which further implies $\|v_1\|^2 \lesssim \|v_2\|^2 \vee (h_1 \vee h_2) \lesssim h_1 \vee h_2$. Then the above observations yield
\begin{align*}
	&\exp\left((c'/2) \left(h_1 \wedge h_2\right)^{-(\beta-d/2-1)}\right) q(x,y) \\
	&\lesssim \ind(\|v_1\|^2, \|v_2\|^2 \lesssim h_2) \Bigg[h_1^{d/2}\int_0^{h_1} \exp\left(-(c'/2) h_3^{-(\beta-d/2-1)}\right) h_3^{-\beta} \md h_3 \\
	& \qquad + \int_{h_1}^{h_2} \exp\left(-(c'/2) h_3^{-(\beta-d/2-1)}\right) h_3^{d/2-\beta} \md h_3\Bigg] +  n^d \int_{h_2 \vee \|v_1\|^2 \vee \|v_2\|^2}^\infty h_3^{-\beta}  \md h_3\\
	& \lesssim (1+h_1^{d/2}) \ind(\|v_1\|^2, \|v_2\|^2 \lesssim h_2)  + n^d (h_2 \vee \|v_1\|^2 \vee \|v_2\|^2)^{-\beta+1}.
\end{align*}
where in the final step, we have used \eqref{eq:gammaint}. Since a similar argument holds when $h_1 >h_2$, for any $h_1, h_2 \ge 0$ and $x,y \in \X'$ with $v_2 \in W_{2\sqrt{d}n}^c$, we have
\begin{multline*}
	\exp\left((c'/2) \left(h_1 \wedge h_2\right)^{-(\beta-d/2-1)}\right) q(x,y) \\
	\lesssim [1+(h_1 \wedge h_2)^{d/2}] \ind(\|v_1\|^2, \|v_2\|^2 \lesssim h_1 \vee h_2)  + n^d (h_1 \vee h_2 \vee \|v_1\|^2 \vee \|v_2\|^2)^{-\beta+1}.
\end{multline*}
Combining the cases for $v_2 \in W_{2\sqrt{d}n}$ and $v_2 \in W_{2\sqrt{d}n}^c$, we thus have for any $y \in \X'$ that
\begin{align}\label{eq:f3split}
    &f_\alpha^{(3)}(y) \lesssim  \ind(v_2 \in W_{2\sqrt{d}n}) \Bigg[ \int (h_1 \vee h_2)^{\frac{(\alpha+1)d}{2}} \exp\left(-(\alpha c'/2) \left(h_1 \wedge h_2\right)^{-(\beta-d/2-1)}\right)(1+ h_1^{\frac{21d}{10}})h_1^{-\beta}  \md h_1 \nonumber\\
    &\quad +\int (h_1 \vee h_2)^{(\alpha+1)d/2-\alpha \beta+\alpha} h_1^{-\beta} \exp\left(-(\alpha c'/2) \left(h_1 \wedge h_2\right)^{-(\beta-d/2-1)}\right)(1+ h_1^{21d/10}) \md h_1 \nonumber\\
    &\quad + \int \ind(\|v_1 - v_2\|^2 \gtrsim h_1 \vee h_2)(\|v_1 - v_2\|^2)^{-\alpha(\beta-d/2-1)} h_1^{-\beta}\nonumber\\
    &\qquad \qquad \qquad \qquad \times \exp\left(-(\alpha c'/2) \left(h_1 \wedge h_2\right)^{-(\beta-d/2-1)}\right)(1+ h_1^{21d/10}) \Q'(\md x) \Bigg] \nonumber\\
    &\quad +\ind(v_2 \in W_{2\sqrt{d}n}^c) \Bigg[\int \ind(\|v_1\|^2, \|v_2\|^2 \lesssim h_1 \vee h_2) [1+(h_1 \wedge h_2)^{\alpha d/2}] \nonumber\\
    & \qquad \qquad \qquad \qquad\times \exp\left(-(\alpha c'/2) \left(h_1 \wedge h_2\right)^{-(\beta-d/2-1)}\right)(1+ h_1^{21d/10}) \Q'(\md x) \nonumber\\
    & + n^{\alpha d}\int (h_1 \vee h_2 \vee \|v_1\|^2 \vee \|v_2\|^2)^{\alpha(-\beta+1)} \exp\left(-(\alpha c'/2) \left(h_1 \wedge h_2\right)^{-(\beta-d/2-1)}\right)(1+ h_1^{21d/10}) \Q'(\md x)\Bigg].
\end{align}
For simplicity, we bound the integral of $f_{\alpha}^{(3)}$ for the two cases $v_2 \in W_{2\sqrt{d}n}$ and $v_2 \in W_{2\sqrt{d}n}^c$ separately below.

\noindent \underline{Case 1 ($v_2 \in W_{2\sqrt{d}n}$):} Using \eqref{eq:intinwn}, for $\beta-1 >( \frac{21}{5} + (\alpha+1)) d/2$ (sufficient to have $\beta>3d+1$, since $\alpha \le 2/5$), the first summand on the r.h.s.\ in \eqref{eq:f3split} can be bounded (up to constants) by
$$
\exp\left(-c'' h_2^{-(\beta-d/2-1)}\right) ( h_2^{(\alpha+1)d/2} + 1).
$$
Similarly, since $\beta-1>(\alpha+1)d/2-\alpha \beta+ \alpha+ (21/10)d$ for $\beta>3d+1$, applying \eqref{eq:intinwn} shows that the second summand in \eqref{eq:f3split} is bounded (up to constants) by
$$
\exp\left(-c'' h_2^{-(\beta-d/2-1)}\right) ( h_2^{(\alpha+1)d/2-\alpha \beta+\alpha} + 1).
$$
Hence combined together, the sum of the first two summands in \eqref{eq:f3split} is bounded (up to constants) by $$\exp\left(-c'' h_2^{-(\beta-d/2-1)}\right) ( h_2^{(\alpha+1)d/2-\alpha \beta+\alpha} + h_2^{(\alpha+1)d/2} + 1) \lesssim \exp\left(-c'' h_2^{-(\beta-d/2-1)}\right) ( h_2^{(\alpha+1)d/2} + 1).
$$

Finally, for the third summand in \eqref{eq:f3split}, integrating over $v_1$, since $\beta>3d+1$, we note that
\begin{multline}\label{eq:alpha}
\int \ind(\|v_1 - v_2\|^2 \gtrsim h_1 \vee h_2)(\|v_1 - v_2\|^2)^{-\alpha(\beta-d/2-1)} \md v_1 \lesssim \int_{\sqrt{h_1 \vee h_2}}^\infty r^{-2\alpha(\beta-d/2-1)} r^{d-1} \md r\\
= \int_{\sqrt{h_1 \vee h_2}}^\infty r^{-2\alpha(\beta-d/2-d/(2\alpha)-1) -1} \md r \lesssim (h_1 \vee h_2)^{\alpha\left(\frac{\alpha+1}{2\alpha}d+1-\beta\right)}.
\end{multline}
Thus, noting that $\beta-1>3d \ge (\alpha+1)d /(2\alpha)$ and bounding this term again using \eqref{eq:intinwn} with $(\alpha+1)d/2-\alpha\beta+\alpha+21d/10<\beta-1$ for $\beta>3d+1$ and $1/5\leq \alpha\leq 2/5$, the overall contribution to $f_\alpha^{(3)}(y)$ in \eqref{eq:f3split} when $v_2 \in W_{2\sqrt{d}n}$ is bounded (up to constants) by
\begin{multline*}
     \exp\left(-c'' h_2^{-(\beta-d/2-1)}\right) ( h_2^{\frac{(\alpha+1)d}{2}} + h_2^{\alpha\left(\frac{\alpha+1}{2\alpha}d+1-\beta\right)}+ 1) \\
     \lesssim \exp\left(-c'' h_2^{-(\beta-d/2-1)}\right) ( h_2^{\frac{(\alpha+1)d}{2}} + h_2^{\alpha\left(\frac{\alpha+1}{2\alpha}d+1-\beta\right)}),
\end{multline*}
where we have used that since the two exponents of $h_2$ in the bound above are of opposite sign, one of these two terms dominate the sum depending on the value of $h_2 \in (0,\infty)$.
Thus, upon integration and using \eqref{eq:gammaint} noting $\beta>3d+1$, we obtain
\begin{equation}\label{eq:b'31}
    \int_{W_{2\sqrt{d}n} \times (0,\infty)} f_\alpha^{(3)}(y)\Q'(\md y),  \int_{W_{2\sqrt{d}n} \times (0,\infty)} f_\alpha^{(3)}(y)^2\Q'(\md y) \lesssim n^d.
\end{equation}

\noindent \underline{Case 2 ($v_2 \in W_{2\sqrt{d}n}^c$):} For the fourth summand on the r.h.s.\ of the bound in \eqref{eq:f3split}, integrating w.r.t.\ $v_1$ in the first step, we obtain for $\beta>3d+1$
\begin{align*}
	&\int \ind(\|v_2\|^2 \lesssim h_1 \vee h_2) (h_1 \vee h_2)^{d/2} (1+(h_1 \wedge h_2)^{\alpha d/2})\\
	& \qquad \qquad \qquad \times \exp\left(-(\alpha c'/2) \left(h_1 \wedge h_2\right)^{-(\beta-d/2-1)}\right)(1+ h_1^{21d/10}) h_1^{-\beta} \md h_1\\
	& \lesssim h_2^{d/2} \ind(\|v_2\|^2 \lesssim h_2) \int_0^{h_2} \exp\left(-(\alpha c'/2) h_1^{-(\beta-d/2-1)}\right)(1+ h_1^{21d/10}) (1+ h_1^{\alpha d/2}) h_1^{-\beta} \md h_1\\
	& \qquad \qquad \qquad +(1+h_2^{\alpha d/2}) \exp\left(-(\alpha c'/2) h_2^{-(\beta-d/2-1)}\right) \int_{\|v_2\|^2 \vee h_2}^\infty (1+ h_1^{21d/10}) h_1^{-\beta + d/2} \md h_1\\
	& \lesssim \exp\left(-c'' h_2^{-(\beta-d/2-1)}\right) \Bigg[h_2^{d/2}\ind(\|v_2\|^2 \lesssim h_2) + (1+h_2^{\alpha d/2})  (\|v_2\|^2 \vee h_2)^{-(\beta-13d/5-1)}\Bigg]\\
	& \lesssim \exp\left(-c'' h_2^{-(\beta-d/2-1)}\right) \Bigg[h_2^{d/2} \ind(\|v_2\|^2 \lesssim h_2) + \ind(\|v_2\|^2 \gtrsim h_2) (1+h_2^{\alpha d/2})  \|v_2\|^{-2(\beta-13d/5-1)}\Bigg],
\end{align*}
where in the penultimate step, in addition to \eqref{eq:gammaint}, we have also used for the second integral that $h_1 \ge \|v_2\|^2 \ge 1$. 

Finally, we consider the last summand in \eqref{eq:f3split}, i.e.,
\begin{align*}
	&n^{\alpha d}\int (h_1 \vee h_2 \vee \|v_1\|^2 \vee \|v_2\|^2)^{\alpha(-\beta+1)} \exp\left(-(\alpha c'/2) \left(h_1 \wedge h_2\right)^{-(\beta-d/2-1)}\right)(1+ h_1^{21d/10}) \Q'(\md x).
\end{align*}
Now, splitting the integral into two parts, first for the case $\|v_1\|^2 \le h_1 \vee h_2 \vee \|v_2\|^2$ and the second one for the case $\|v_1\|^2 > h_1 \vee h_2 \vee \|v_2\|^2$, we can upper bound the integral above by 
\begin{align*}
	& n^{\alpha d} \Bigg[\int (h_1 \vee h_2\vee \|v_2\|^2)^{d/2 + \alpha(-\beta+1)} \exp\left(-(\alpha c'/2) \left(h_1 \wedge h_2\right)^{-(\beta-d/2-1)}\right)(1+ h_1^{21d/10}) h_1^{-\beta} \md h_1\\
	&+ \int (h_1 \vee h_2\vee \|v_2\|^2)^{1/5(5d/2+1-\beta)}\exp\left(-(\alpha c'/2) \left(h_1 \wedge h_2\right)^{-(\beta-d/2-1)}\right)(1+ h_1^{21d/10}) h_1^{-\beta} \md h_1\Bigg]\\
	& \lesssim n^{\alpha d}  \int (h_1 \vee h_2\vee \|v_2\|^2)^{1/5(5d/2+1-\beta)}\exp\left(-(\alpha c'/2) \left(h_1 \wedge h_2\right)^{-(\beta-d/2-1)}\right)(1+ h_1^{21d/10}) h_1^{-\beta} \md h_1,
\end{align*}
where to obtain the second summand above when $\|v_1\|^2 > h_1 \vee h_2 \vee \|v_2\|^2$, we have used (argument similar to \eqref{eq:alpha} noting $\alpha>1/5 $ and $\beta>5d/2+1$) that
\begin{multline*}
\int \ind(\|v_1\| \ge \sqrt{h_1 \vee h_2 \vee \|v_2\|^2}) \|v_1\|^{2\alpha(-\beta+1)} \md v_1 \\
 \lesssim \int_{\sqrt{h_1 \vee h_2 \vee \|v_2\|^2}}^\infty r^{-\frac{2}{5}(\beta-5d/2-1) -1} \md r \lesssim (h_1 \vee h_2 \vee \|v_2\|^2)^{\frac{1}{5}(5d/2+1-\beta)}.
\end{multline*}
Further, using \eqref{eq:intinwn}and \eqref{eq:gammaint}, we can bound
\begin{align*}
	& n^{\alpha d}  \int (h_1 \vee h_2\vee \|v_2\|^2)^{1/5(5d/2+1-\beta)}\exp\left(-(\alpha c'/2) \left(h_1 \wedge h_2\right)^{-(\beta-d/2-1)}\right)(1+ h_1^{21d/10}) h_1^{-\beta} \md h_1\\
	& \lesssim n^{\alpha d} \exp\left(-c'' h_2^{-(\beta-d/2-1)}\right)\Bigg[\ind(\|v_2\|^2 \le h_2) (1+h_2^{1/5(5d/2+1-\beta)}) + \ind(\|v_2\|^2 > h_2) \|v_2\|^{2/5(5d/2+1-\beta)}\Bigg]\\
	& \lesssim n^{\alpha d} \exp\left(-c'' h_2^{-(\beta-d/2-1)}\right)\Bigg[\ind(\|v_2\|^2 \le h_2) + \ind(\|v_2\|^2 > h_2) \|v_2\|^{2/5(5d/2+1-\beta)}\Bigg].
\end{align*}
Thus we obtain that the total contribution to $f_\alpha^{(3)}(y)$ in \eqref{eq:f3split} when $v_2 \in W_{2\sqrt{d}n}^c$ is bounded (up to constants) by
\begin{align}\label{eq:integ5d}
	&\exp\left(-c'' h_2^{-(\beta-d/2-1)}\right) \Bigg[h_2^{d/2} \ind(\|v_2\|^2 \lesssim h_2) + \ind(\|v_2\|^2 \gtrsim h_2) (1+h_2^{\alpha d/2})  \|v_2\|^{-2(\beta-13d/5-1)}\Bigg] \nonumber\\
	&\quad + n^{\alpha d} \exp\left(-c'' h_2^{-(\beta-d/2-1)}\right)\Bigg[\ind(\|v_2\|^2 \le h_2) + \ind(\|v_2\|^2 > h_2) \|v_2\|^{2/5(5d/2+1-\beta)}\Bigg].
\end{align}
Thus, we can bound $\Q' f_\alpha^{(3)}$ when $v_2 \in W_{2\sqrt{d}n}^c$ and $\beta>5d+1$ by
\begin{align*}
    &\int (h_2^{d-\beta}+n^{\alpha d}h_2^{d/2-\beta}) \exp\left(-c'' h_2^{-(\beta-d/2-1)}\right) \md h_2 \\
    &\qquad + \int (1+h_2^{\alpha d/2}) h_2^{-(\beta -21d/10 -1)} \exp\left(-c'' h_2^{-(\beta-d/2-1)}\right) h_2^{-\beta} \md h_2\\
    & \qquad \qquad +n^{\alpha d}\int h_2^{-\frac{1}{5}(\beta -5d -1)} \exp\left(-c'' h_2^{-(\beta-d/2-1)}\right) h_2^{-\beta}\md h_2 \lesssim 1+n^{\alpha d},
\end{align*}
where we have used again that for $a>d$,
$$
\int_{\|v_2\| \gtrsim \sqrt{h_2}} \|v_2\|^{-a} \lesssim \int_{\sqrt{h_2}}^\infty r^{d-a-1} \md r \lesssim h_2^{\frac{d-a}{2}}.
$$
The argument to bound $\Q' (f_\alpha^{(3)})^2$ when $v_2 \in W_{2\sqrt{d}n}^c$ and $\beta>5d+1$ can be argued similarly to obtain in this case the bound
\begin{align*}
    &\int (h_2^{3d/2-\beta}+n^{2\alpha d}h_2^{d/2-\beta}) \exp\left(-2c'' h_2^{-(\beta-d/2-1)}\right) \md h_2 \\
    &\qquad + \int (1+h_2^{\alpha d}) h_2^{-(\beta -57d/20 -1)} \exp\left(-2c'' h_2^{-(\beta-d/2-1)}\right) h_2^{-\beta} \md h_2\\
    & \qquad \qquad +n^{2\alpha d}\int h_2^{-\frac{1}{5}(\beta -5d -1)} \exp\left(-2c'' h_2^{-(\beta-d/2-1)}\right) h_2^{-\beta}\md h_2 \lesssim 1+n^{2\alpha d}.
\end{align*}
Combining these bounds with \eqref{eq:b'31}, for $\beta>5d+1$, we obtain 
\begin{equation}\label{eq:b'3}
    \mathbb{Q}' f_{\alpha}^{(3)}, \mathbb{Q}' (f_{\alpha}^{(3)})^2 \lesssim n^d.
\end{equation}

We next bound $\mathbb{Q} \kappa^\alpha$. From Proposition \ref{lm:tailbounds}, for $x=(v_1,h_1)$ with $h_1 <0$, we have
\begin{align*}
	\kappa(x)&=\P(\xi_n((v_1,h_1),\eta_\beta' \cup \{(v_1,h_1)\}) \neq 0) \\
	&\le \ind(v_1 \in W_n)\P(T((v_1,h_1),\eta_\beta' \cup \{(v_1,h_1)\}) >h_1)\\
	&\lesssim \ind(v_1 \in W_n) \exp\big(-c(-h_1)^{-(\beta-d/2-1)}\big).
\end{align*}
Thus, noting that $G(x) \lesssim 1+ h_1^{21d/10}$, integration using \eqref{eq:gammaint} yields $\mathbb{Q} \kappa^\alpha G\lesssim n^d$ for $\beta>21d/10+1$. Noting that $g \lesssim f_\zeta^{(1)}$, and that $\alpha \in (0,1)$, we can argue as for $\mathbb{Q} (f_\alpha^{(1)})^2$ in \eqref{eq:b'12} to conclude that $\mathbb{Q} g^\alpha G\lesssim n^d$ for $\beta>5d+1$. Thus, by Theorem \ref{thm:KolBd},  $\Var(F_n(\eta_\beta)) \lesssim n^d$ for $\beta>5d+1$.

On the other hand, by Proposition \ref{prop:VarianceBound} there exists a constant $C\in (0,\infty)$ depending only on $d, \beta$ such that for all $n \in \N$,
\begin{equation}\label{eq:Varbdb'}
	\operatorname{Var} F_n(\eta_\beta') \ge C n^{d}.
\end{equation}
This proves the first assertion. The second assertion follows by Theorem \ref{thm:KolBd} upon plugging in the the bounds above as well as those in \eqref{eq:b'11}, \eqref{eq:b'12}, \eqref{eq:b'2} and \eqref{eq:b'3}.

The optimality of the bound on the Kolmogorov distance can be argued as in the proof of \cite[Theorem~1.1, Eq.~(1.6)]{EG81}, which shows that the Kolmogorov distance between any integer-valued random variable, normalized, and a standard normal random variable is always lower bounded by a constant times the inverse of the standard deviation (see Section~6 therein for further details). The variance upper bounds in Theorem \ref{thm:1} now yields the result.

\section*{Acknowledgments}

AG and CB were supported by the DFG priority program SPP 2265 \textit{Random Geometric Systems}. AG was additionally supported by Germany's Excellence Strategy  EXC 2044 -- 390685587, \textit{Mathematics M\"unster: Dynamics - Geometry - Structure}.

\end{document}